\documentclass[12pt,reqno]{amsart}
\usepackage{soul}
\usepackage{hyperref}
\usepackage{multicol,amsmath,graphics,cancel,soul,color,bbm,amsfonts,dsfont,tikz,mathrsfs,amssymb,mathtools,bm,bbm,comment}
\usepackage[left=1in, right=1in, top=1.1in,bottom=1.1in]{geometry}
\usetikzlibrary{matrix,patterns,positioning}
\numberwithin{equation}{section}

\newcommand{\E}{\mathbb{E}}

\newcommand{\N}{\mathbb{N}}

\newcommand{\Q}{\mathbb{Q}}
\newcommand{\R}{\mathbb{R}}

\newcommand{\Indi}[1]{\mathbbm{1}_{#1}}

\newcounter{dummy} \numberwithin{dummy}{section}

\newtheorem{Proposition}[dummy]{Proposition}

\newtheoremstyle{DefinitionStyle}  % name of the style to be used
	{9pt}				% measure of space to leave above the theorem. E.g.: 3pt
	{9pt}				% measure of space to leave below the theorem. E.g.: 3pt
	{}					% name of font to use in the body of the theorem
	{0pt}				% measure of space to indent
	{\bfseries}		    % name of head font
	{.}					% punctuation between head and body
	{3pt}				% space after theorem head
	{}					% Manually specify head
\theoremstyle{DefinitionStyle}

\newtheorem{lemma}{Lemma}

\newtheorem{theorem}{Theorem}

\newtheorem{remark}{Remark}

\def\1{{\rm l}\hskip -0.21truecm 1}

\newcommand{\vertiii}[1]{{\left\vert\kern-0.25ex\left\vert\kern-0.25ex\left\vert #1 
    \right\vert\kern-0.25ex\right\vert\kern-0.25ex\right\vert}}
\def\R{\mathbb{R}}

\def\Var{\mathbb{V}\mathrm{ar}}

\def\dh2l{\mathbf{d}_{\mathbb{H}_{2\ell}}}
\def\d2{\mathbf{d}_2}

\begin{document}

\title[Asymptotics of particle systems]{Asymptotics for additive functionals of particle systems via Stein's method}
\author{Arturo Jaramillo and Antonio Murillo-Salas}
\address{Arturo Jaramillo: Centro de Investigaci\'on en Matem\'aticas, Jalisco S/N, Col. Valenciana 36023 Guanajuato, Gto.}
\email{jagil@cimat.mx}

\address{Antonio Murillo-Salas: Universidad de Guanajuato, Guanajuato, Mexico.}
\email{amurillos@ugto.mx}
\date{\today}

\maketitle

\begin{abstract}
We consider additive functionals of systems of random measures whose initial configuration is given by a Poisson point process, and whose individual components evolve according to arbitrary Markovian or non-Markovian measure-valued dynamics, with no structural assumptions beyond basic moment bounds. In this setting and under adequate conditions, we establish a general third moment theorem for the normalized functionals. Building on this result, we obtain the first quantitative bounds in the Wasserstein distance for a variety of moving-measure models initialized by Poisson-driven clouds of points, turning qualitative central limit theorems into explicit rates of convergence. The scope of the approach is then demonstrated through several examples, including systems driven by fractional Brownian motion, $\alpha$-stable processes, uniformly elliptic diffusions, and spectral empirical measures arising from Dyson Brownian motion, all under broad assumptions on the control measure of the initial Poisson configuration. The analysis relies on a combination of Stein's method with Mecke's formula, in the spirit of the Poisson Malliavin-Stein methodology.\\

\noindent{\bf MSC 2020:} 	60G55,60G22, 60G52, 60B10.	

\noindent{\bf Key words and phrases:} {\it particle systems, occupation-time fluctuations, limit theorems, Stein's method.}
\end{abstract}

\section{Introduction}\label{sec:intro}
Understanding the behavior of time-evolving systems of random measures plays a central role in modern probability. Systems of this sort arise in many areas; for instance, in statistical physics through measure-valued limits \cite{Dawson1993,Spohn2012}, in population dynamics via superprocesses \cite{Etheridge2000,Perkins2002}, and in random matrix theory through Dyson's interacting-eigenvalue model \cite{AndersonGuionnetZeitouni2010}. This paper contributes to this line of work by considering systems whose initial state is randomized through a Poisson point process. A natural way to build such systems is to start from a Poisson point process that fixes the initial configuration and then develop an adequately chosen stochastic measure evolution (see Section \ref{sec:modeldesc} for details). General references for constructions of this type appear in the theory of random measures \cite{Kallenberg2017} and in the classical framework of point processes \cite{DVJ2008}. This perspective is consistent with the general framework of measure-valued processes developed in \cite{Dawson1993} and with the treatment of superprocesses and measure-valued diffusions in \cite{Perkins2002}.\\

We focus exclusively on the additive functionals generated by this Poisson-driven evolution, a class of objects that fits naturally within the classical framework developed in \cite{Dynkin1991,Dawson1993}. These functionals capture global quantities of interest, such as occupation times, empirical measures, or aggregated interactions across the system, with related analyses carried out in \cite{C-G,D-W,BGT-I,BGT-II,LM-MS-RG}.\\

\noindent
An archetypal version of the model mentioned above can be formulated as follows. Let $\mathbb{X}$ be a Poisson random configuration, which we regard as the random parameter specifying the initial condition of a time-dependent system of measures.  Fix a Radon measure $\nu$ on $\mathbb{X}$, and let $\eta$ be a Poisson random measure with intensity $\nu$. For each point $x \in \mathbb{X}$, we associate a measure-valued process $(\mu_t^{x})_{t \ge 0}$ on a second measurable space $\mathbb{Y}$, with the processes evolving independently over time, such evolution can be quite general as shown in Sections \ref{sec:particle:systems} and \ref{sec:spectralempirical}. Given a bounded test function $\psi$ on $\mathbb{Y}$ with compact support, we  consider the additive functional
\[
A_r[\psi]
   := 
   \int_0^r \int_{\mathbb{X}} \psi(y)\,\mu_t^{x}(dy)\,\eta(dx)\,dt,
\]
which aggregates the contributions of all the particles along the time interval $[0,r]$.  Our goal is to understand the asymptotic behavior of the normalized sequence
\begin{align}\label{eq:Zndef}
Z_r[\psi] := \frac{1}{\sigma_r} \left(A_r[\psi] - \mathbb{E}[A_r[\psi]]\right),	
\end{align}
where \(\sigma_r\) denotes the standard deviation of \(A_r[\psi]\). We are particularly interested in developing a general framework that both captures the emergence of Gaussian fluctuations for $Z_r[\psi]$ and provides quantitative control over their convergence, with enough flexibility to encompass a broad range of underlying dynamics.\\

\noindent The main contributions of this work are threefold. First, we establish a general third moment theorem for the functional $Z_r[\psi]$, demonstrating that the underlying Poissonian randomness drives the Gaussian fluctuations. Our quantitative bounds are consistent with the fourth moment framework previously developed in Gaussian~\cite{nourdin2009stein,nourdin2012normal,nualart2005central,neufcourt2016third} and Poissonian settings~\cite{dobler2018fourth}, but reveal a less common phenomenon: Gaussian limits can be characterized through control of moments up to order three, rather than four. We show that our third moment theorem opens the way to extending classical analyses of occupation time fluctuations, previously carried out in Poissonian and branching systems using tools such as martingale decompositions, Laplace and Feynman-Kac representations, and Fourier analytic methods in space and time, among others (see \cite{C-G,D-W,BGT-I,BGT-II,LM-MS-RG}). Second, we establish explicit quantitative bounds for \eqref{eq:Zndef} in the Wasserstein distance. To the best of our knowledge, such quantitative normal approximations have not been obtained before in this setting. Third, we illustrate the scope of our framework through examples of measure-valued processes and their additive functionals that go beyond the classical occupation time fluctuation setting, as studied for instance in \cite{C-G,D-W,BGT-I,BGT-II}. These include systems arising from processes with long-range dependence such as fractional Brownian motion, heavy-tailed dynamics like $\alpha$-stable processes, general diffusions with nontrivial drift and volatility, and spectral measures arising in random matrix theory, with Dyson Brownian motion serving as a first illustrative example.\\

\noindent
Our methodology combines Stein's method,  Mecke's formula, together with a careful decomposition of the different sources of randomness in the model. The procedure can be summarized as follows. (i) The first step is to rewrite the system so that one can clearly see which part of the randomness comes from the initial Poisson configuration and which part stems from the measure-valued dynamics. (ii) Once this distinction is in place, we implement Stein's method, reducing the problem to bounding the expectation of the Stein operator applied to a suitable test function, and evaluated in the normalized random variables under consideration. (iii) To control this former term, we apply an integration by parts formula to the Poissonian component of the system by means of an application of  Mecke's formula. This yields a bound for the quantity of interest that depends only on the deviation away from zero for the third moment of the normalized functional. (iv) The remaining task is to obtain sharp bounds on the third moment of the normalized functional. This is achieved through explicit computations that depend on suitable control of the conditional densities, up to order three, of the processes driving the non-Poissonian dynamics. In several examples, these computations lead to expressions closely related to the variance estimates for occupation times of particle systems driven by $\alpha$-stable processes, as studied by Bojdecki, Gorostiza, and Talarczyk in \cite{BGT-I,BGT-II}. These Markovian estimates play a crucial role in making our bounds explicit, and their analytical structure remains effective even when the underlying driving processes cease to be Markovian.\\

\noindent While the methodology can accommodate more intricate dynamics, including systems with branching mechanisms, in this paper we restrict attention to non-branching measure evolutions. This particular setting already yields nontrivial fluctuation results and illustrates the scope of our approach.  Before turning to the specifics of our main results, we first lay out the structure of the model and fix some minimal notation that will be used throughout. At this stage we restrict ourselves to describing the Poissonian input, the associated evolving measures, and the additive functionals of interest. Broader preliminaries and auxiliary tools will be introduced later, once they are actually needed for the analysis.\\

\noindent The remainder of the paper is organized as follows.  In Section \ref{sec:modeldesc}, we introduce the  framework for our model, separating the randomness of the initial Poisson configuration from that of the  measure-valued evolutions. We then give the precise definition of the additive functionals under consideration and present several fundamental examples.  Section \ref{se:Preliminariessec} presents the main preliminaries, including the elements of Stein's method, as well as  Mecke's formula.  In Section \ref{sec:abstractmomenthm}, we establish our abstract moment theorems, deriving quantitative Gaussian approximation bounds in terms of third moments (and, in a more general formulation, fourth moments).  Section \ref{sec:particle:systems} contains the first set of applications, treating Poissonian particle systems in a broad range of regimes, including $\alpha$-stable motions, uniformly elliptic diffusions, and non-Markovian models such as particles driven by fractional Brownian paths, and deriving explicit Wasserstein rates in each case. 
Section \ref{sec:spectralempirical} applies our methodology to the analysis of the spectral empirical measures generated by the Dyson Brownian motion. In Section \ref{sec:secondhtoemre} we present the proof of the main application theorem. Finally, in the Appendix  we prove the auxiliary variance estimates and density bounds used throughout the paper.

\section{Model description}\label{sec:modeldesc}
Our construction proceeds in three steps. First, we describe the Poissonian random parameter space that encodes the initial particle configuration of the system. Next, we specify the family of measure-valued processes that evolve from this random input. Finally, we define the additive functionals that will serve as the main objects of our study.

\subsection{Poisson random parameter space} Let \((\mathbb{X}, \mathcal{X})\) be a measurable space representing the domain in which the Poisson configuration points are sampled. Denote by \(\mathbf{N}_{\sigma}(\mathbb{X})\) the space of positive integer-valued, Radon measures on \(\mathbb{X}\), equipped with the smallest \(\sigma\)-algebra \(\mathcal{N}_{\sigma}(\mathbb{X})\) that makes the mappings \(\chi \mapsto \chi(B)\) measurable for every \(B\) in $\mathcal{X}$. Elements of \(\mathbf{N}_{\sigma}(\mathbb{X})\) will play the role of stochastic parameters that encode the initial particle configurations of the system.\\

\noindent To introduce randomness into these configurations, we fix a Radon control measure \(\nu\) on \((\mathbb{X}, \mathcal{X})\), and define a probability measure \(\vartheta\) on \(\mathcal{N}_{\sigma}(\mathbb{X})\) such that a random measure \(\eta\) is distributed according to \(\vartheta\) if, for any disjoint finite collection \(B_1, \dots, B_n \) in $\mathcal{X}$, the random variables \(\eta(B_1), \dots, \eta(B_n)\) are, independent and each \(\eta(B_i)\) follows a Poisson distribution with mean \(\nu(B_i)\). In other words, \(\eta\) is a Poisson random measure in \( \mathbb{X}\) with intensity \(\nu\). The measure $\nu$ is assumed to be absolutely continuous with respect to the Lebesgue measure with a density bounded from above and away form zero.\\

\subsection{Measure valued processes} Having described the parameter process, we now specify the evolution of the system. Consider an additional measurable space \((\mathbb{Y}, \mathcal{Y})\), which will serve as the ambient space (or state space) for the system, and let \(\mathcal{M}_{\sigma}(\mathbb{Y})\) denote the space of \(\sigma\)-finite measures on \((\mathbb{Y}, \mathcal{Y})\). Let \(\{\mu^x \; ; \; x \in \mathbb{X}\}\) be a family of independent \(\mathcal{M}_{\sigma}(\mathbb{Y})\)-valued processes, defined on a probability space \((\Omega, \mathcal{F}, \mathbb{Q})\), each of the form \(\mu^x = \{\mu_t^x \; ; \; t \geq 0\}\). Given an initial configuration \(\eta\) in \(\mathcal{N}_{\sigma}(\mathbb{X})\), we define the induced \(\mathcal{M}_{\sigma}(\mathbb{Y})\)-valued process \(\mu^{\eta} = \{\mu_t^{\eta} \; ; \; t \geq 0\}\) by
\[
\mu_t^{\eta}(\omega) := \int_{\mathbb{X}} \mu_t^x(\omega) \, \eta(dx),  \,\, \omega \in\Omega.
\]
 For each fixed \(x\) in $ \mathbb{X}$, the measure \(\mu_t^x(\omega)\) will be viewed as a function \(A \mapsto \mu_t^x(\omega, A)\) on \(\mathcal{Y}\). The process of interest is the \(\mathcal{M}_{\sigma}(\mathbb{Y})\)-valued process $\mu:=\{\mu_t^\eta(\omega)\,:\, (\eta,\omega)\in\tilde{\Omega}, t\geq0\}$ where the product space \(\tilde{\Omega} := \mathcal{N}_{\sigma}(\mathbb{X}) \times \Omega\) is equipped with the product \(\sigma\)-algebra \(\mathcal{N}_{\sigma}(\mathbb{X}) \otimes \mathcal{F}\) and the product measure \(\mathbb{P} := \vartheta \otimes \mathbb{Q}\). \\

\subsection{Additive functionals}  We now define the additive functionals of interest. For a bounded test function \(\psi: \mathbb{Y} \to \mathbb{R}\) with compact support and $r\in(0,\infty)$, we define the functional \(A_r[\psi]\) over $\tilde{\Omega}$, as
\begin{align}\label{eq:Ardef}
A_r[\psi](\eta, \omega) := \int_0^r \langle \mu_t^{\eta}(\omega), \psi \rangle \, dt, 
\end{align}
where we have used the dual pairing notation 
\begin{align*}
\langle \rho, \psi \rangle
  &:= \int_{\mathbb{X}} \psi(x) \, \rho(dx),\,\,\rho\in\mathcal{N}_{\sigma}(\mathbb{X}) .
\end{align*}
The quantity \(A_r[\psi]\) encodes the cumulative interaction of the evolving system through the lens of the spatial filter \(\psi\), up to time \(r\). Our goal is to study the distributional behavior of the normalized sequence
\begin{equation} \label{mainZ}
Z_r[\psi] := \frac{1}{\sigma_r} \left(A_r[\psi] - \mathbb{E}[A_r[\psi]]\right),
\end{equation}
where \(\sigma_r\) denotes the standard deviation of \(A_r[\psi]\) and $\E$ denotes the expectation with respect to the probability measure $\mathbb{P}$.\\

\begin{comment}
\subsection{Fundamental examples} Although the notation introduced here may appear cumbersome, it offers the advantage of encompassing a broad class of stochastic models, including interacting particle systems, branching processes, spectrum of random matrix models, Markov diffusions, among others. We anticipate to the reader that, while we describe the basic structure shared by such models and present a universal methodology potentially applicable to all of them, the scope of this paper remains introductory. The path towards treating other examples (such as examples exhibiting ramification) is conceptually clear, but demands a substantial technical effort that remains to be explored.\\ 
\end{comment}

\subsection{Fundamental examples}
Although the notation may seem cumbersome at first, it is sufficiently flexible to encompass a wide variety of stochastic models, including interacting particle systems, branching processes, random matrix spectra, and Markov diffusions. The present work, however, is intended as an introduction to the general structure underlying these examples. Extensions to more elaborate settings, such as models with ramification, are conceptually accessible but require substantial additional analysis, which we do not pursue here. Next we present some examples of models arising from specific choices of the configuration $\eta$ and the evolving measures $\mu^{\eta}$.\\

\noindent\textit{Dirac particle systems}\\
Take $\mathbb{X}=\R^{d}$ and $\mathbb{Y}=\mathbb{X}^{d}$. Consider a collection $\mathbf{X}=\{\mathbf{X}^{x}\ ;\ x\in\mathbb{X}\}$ of independent processes over $\mathbb{Y}$ of the form $\mathbf{X}^{x}=\{\mathbf{X}_t^{x}\ ;\ t\geq 0\}$, with $\mathbf{X}_t^{x}=({X}_t^{1,x},\dots, {X}_t^{m,x})$. We will assume that $X_0^{i,x}=x\in\R^d$. Then, we define 
\begin{align}\label{eq:mutdefsumdeltas}
\mu_t^{x}
  &:=a_1\delta_{X_t^{1,x}}+\cdots+a_m\delta_{X_t^{m,x}},
\end{align}
for a given choice of strictly positive coefficients $a_1,\dots, a_m$. The analysis of these processes will be carried out in detail in Section \ref{sec:particle:systems}. The dynamics of $\mathbf{X}^x$, can be quite general. For instance, they may arise from an interacting particle system governed by the stochastic differential equation  
\[
d\mathbf{X}_{t}^{x}
  = U(\mathbf{X}_{t}^{x})\,dt + V(\mathbf{X}_{t}^{x})\,d\mathbf{W}_{t},
\]
where $\mathbf{W}_{t}$ is an $\ell$-dimensional Brownian motion on $\mathbb{R}^{\ell}$,  and the coefficients $U$ and $V$ are functions from $\mathbb{R}^{md}$ to $\mathbb{R}^{md}$ and $\mathbb{R}^{\ell \times md}$, respectively. 
Beyond this Markovian setting, the framework also allows for non-Markovian dynamics, such as fractional Brownian motion or $\alpha$-stable processes on $\mathbb{R}^{md}$. It is worth mentioning that the spherically symmetric $\alpha$-stable motion case where $m=1$ and $a_1=1$ was treated in \cite{BGT-I} and \cite{BGT-II}, where functional central limit theorems for this kind of objects where obtained by means of an analysis of the moments. Different scenarios are considered in Section \ref{sec:particle:systems}.\\

%\noindent Our next example concerns the analysis of random matrix processes, and will be presented in detail in Section \ref{sec:spectralempirical}.\\ 

\noindent\textit{Spectral empirical distributions}\\
Take $\mathbb{X}$ as the set of symmetric real valued matrices of dimension $d$ and $\mathbb{Y}=\R$. We let $\{\mathbf{Y}^x\ ;\ x\in\mathbb{X}\}$ be a collection of independent matrix valued processes. Consider the $\R^d$-valued process of eigenvalues $(\lambda_t^{1,x},\dots, \lambda_t^{d,x})$. We then associate the empirical distribution
$$\mu_t^{x}:=a_1\delta_{\lambda_t^{1,x}}+\cdots+a_d\delta_{\lambda_t^{d,x}}.$$ 
This model will we treated in detail in Section \ref{sec:spectralempirical}.

%\noindent We conclude with an example that falls outside the scope of the present analysis but fits naturally within the general framework.\\

\noindent\textit{Critical branching particle systems}\\
Take $\mathbb{X}=\mathbb{Y}=\R^{d}$. Consider a collection $\{X^{x}\ ;\ x\in\mathbb{X}\}$ of independent Markov branching processes over $\mathbb{Y}$ of the form $X^{x}=\{X_t^{x}\ ;\ t\geq 0\}$. More precisely, $X_t^{x}$ follows a Markov movement for an exponential time, time at which the particle dies and undergoes a branching procedure, giving rise to a mean-one random number of offspring, placed at its parent position, each of which evolves identically to its progenitor particle. This example falls outside the scope of the present analysis but fits naturally within the general framework presented in this work.\\

\section{Preliminaries}\label{se:Preliminariessec}
In this section, we present some preliminaries on Stein's method and analysis of Poisson functionals, which will be used throughout the paper. In what follows, for every measurable space $(\mathbb{Y},\mathcal{Y})$, we denote by $\mathcal{P}(\mathbb{Y})$ the set of probability measures on $\mathbb{Y}$. In the particular case where $\mathbb{Y} = \mathbb{R}$, we denote by $\gamma$ the standard Gaussian distribution.

\subsection{Stein's method}\label{sec:Stein-intro}
Our quantitative Gaussian approximation results for $Z_r[\psi]$ will be derived using the celebrated Stein's method. At its core, the method relies on the fact that a probability measure $\rho$ on $\R$ is standard normal exactly when
\[
\int_{\R} \mathcal{A}[g](x)\,\rho(dx)=0,
\]
for every differentiable test function $g$ with bounded derivative, where the Stein operator $\mathcal{A}$ is given by
\[
\mathcal{A}[f](x) := f'(x) - x f(x).
\]
Given a real test function $h$ on $\R$, we define the associated solution to the Stein equation as the unique bounded solution to 
\begin{equation}\label{St-eq}
\mathcal{A}[f_h](x)= h(x) - \int_{\R} h(x)\,\gamma(dx).
\end{equation}
Whenever $h$ is Lipschitz, this equation admits a solution  with suitable regularity properties. The key idea is that if a random variable $W$ is close in distribution to a standard Gaussian random variable, then the expectation
\begin{align}\label{eq:Steinopdef}
|\mathbb{E}[f_h'(W) - W f_h(W)]|	
\end{align}
should be small. This quantity serves as a tractable proxy for the difference between $\mathbb{E}[h(W)]$ and $\mathbb{E}[h(Z)]$. To control this difference, one typically solves the Stein equation for a given test function $h$, derives bounds on the solution $f_h$ and its derivatives, and estimates the expectation $|\mathbb{E}[\mathcal{A}[f_h](W)]|$ using known properties of $W$. For more details on the Stein method, the reader can see \cite{nourdin2012normal}.\\

\noindent As previously mentioned, the test functions we will consider for $h$ are the one-Lipschitz functions, denoted by $\mbox{Lip}_1$. In this case, it can be shown that the corresponding solution $f_h$ belongs to the class $\mathcal{C}$ of real functions $g$ over $\R$, satisfying $\|g\|_{\infty} \leq \sqrt{\pi/2}$, $\|g'\|_{\infty} \leq 1$, and $\|g''\|_{\infty} \leq 2$. By letting $h$ range over all functions in $\mbox{Lip}_1$, we recover the one-Wasserstein distance $d_{W}$ defined by
\begin{align}\label{eq:dWdistance}
d_{W}(\rho_1,\rho_2)
  &:= \sup_{h \in \text{Lip}_1} \left| \int_{\R} h(x)\, \rho_1(dx) - \int_{\R} h(x)\, \rho_2(dx) \right|,\,\, \rho_1,\rho_2\in\mathcal{P}(\R).
\end{align}
 Integrating the Stein equation with respect to $\rho$ yields the inequality
\begin{align}\label{eq:Steinineq}
d_{W}(\rho,\gamma)
  &\leq \sup_{f \in \mathcal{C}} \left| \int_{\R} \mathcal{A}[f](x)\, \rho(dx) \right|.
\end{align}
This reasoning reframes the original problem. Rather than dealing directly with the right-hand side of \eqref{eq:dWdistance}, we instead work through its Stein reformulation given by the right-hand side of \eqref{eq:Steinineq}.

\subsection{Mecke's formula}
The right-hand side of \eqref{eq:Steinineq} can be handled using standard tools from the analysis of Poisson point processes, most notably Mecke's formula, which we review next. In preparation for this, we recall that the main random objects in our setting are real-valued functions defined on the space $\mathbf{N}_{\sigma}(\mathbb{X})$.  Consider a function $h : \mathbf{N}_{\sigma}(\mathbb{X}) \times \mathbb{X} \to \mathbb{R}$, integrable with respect to $\vartheta \otimes \nu$. Mecke's formula states that
\begin{align}\label{eq:Mecke}
\int_{\mathbf{N}_{\sigma}} \int_{\mathbb{X}} h(\eta, x)\, \eta(dx)\, \vartheta(d\eta)
  = \int_{\mathbf{N}_{\sigma}} \int_{\mathbb{X}} h(\eta + \delta_x, x)\, \nu(dx)\, \vartheta(d\eta).
\end{align}
This identity will allow us to handle expressions involving the term $\mathbb{E}[W f_h(W)]$ appearing in \eqref{eq:Steinopdef}, for the type of random variables that will be relevant to our study. The reader is referred to \cite{PeccatiReitzner2016} for a comprehensive presentation of applications of Mecke's formula, and to \cite{Mecke1967,LastPenrose2017} for rigorous proofs of the identity in its classical and modern formulations.

\section{Abstract Moment Theorems}\label{sec:abstractmomenthm}
\noindent In this section, we present our main abstract result, which will be the basis for proving central limit theorems for the models described in Section \ref{sec:modeldesc}. The result holds under general assumptions on the particle dynamics, allowing for a flexible choice in the evolution of the measures. In the sequel, $\rho_r$ will denote the probability distribution of the random variable $Z_r[\psi]$ defined in \eqref{mainZ}. Given a nonnegative integer $\kappa$, we will denote by $\mathcal{P}_{\kappa}(\R)$ the set of probability measures on $\R$ with finite moments up to order $\kappa$. For any $\tau$ in $ \mathcal{P}_{\kappa}(\R)$ and $1 \leq j \leq \kappa$, we denote by $m_j[\tau]$ the moment of order $j$ of $\tau$, namely,
\begin{align*}
m_{j}[\tau]
  := \int_{\R} x^{j}\, \tau(dx).
\end{align*}
As anticipated in the introduction, our main abstract result is formulated within the framework of the so-called fourth-moment phenomenology. First, we take a moment to briefly outline this setting.  This line of research originated in the context of normal approximations for sequences of functionals on the Wiener space, where it was shown that convergence in distribution towards the standard Gaussian law within a fixed Wiener chaos, is essentially governed by the convergence of the fourth moment to that of the standard Gaussian distribution. This principle was first made quantitative in \cite{nourdin2009stein}, building upon the seminal work by Nualart and Peccati \cite{nualart2005central}, and later developed systematically across a remarkably broad array of frameworks. We refer the reader to the monograph \cite{nourdin2012normal} for an accessible and comprehensive exposition of these ideas.\\

\noindent Subsequent extensions of the fourth moment paradigm include its non-commutative analogue in free probability \cite{kemp2012wigner}, its adaptation to the setting of infinitely divisible distributions \cite{arizmendi2021convergence}, and its counterpart in the Poisson space \cite{dobler2018fourth}. In essence, the phenomenology consists of considering a class $\mathcal{Q}$ of probability measures on $\mathbb{R}$ and proving that for every element in this class, a suitable probabilistic distance to the Gaussian (for instance, the Wasserstein distance) can be bounded by a constant times the square root of the difference between its fourth moment and three. Surprisingly, there are virtually no known examples in which moments of lower order govern convergence in this fashion. To our knowledge, the only example of a third moment phenomenology appears in \cite{neufcourt2016third}. In this regard, Theorem \ref{thm:mainabstractone} below is especially noteworthy, as it provides not only a novel instance of fourth moment phenomenology but also an example where a third moment bound is available.

\begin{theorem}\label{thm:mainabstractone}
%Suppose that
%\begin{align*}
%\int_{\R^{d}}\E[|u_r^{\psi}(x)|^3]\nu(dx)<\infty.
%\end{align*}
%Then it holds that 
%\begin{align}\label{eq:dWmainzero}
%d_{W}(\rho_r,\gamma)
%  \leq 	\sqrt{2}\int_{\R^{d}}\E[|u_r^{\psi}(x)|^3]\nu(dx).
%\end{align}
%Moreover, under the condition $\psi\geq 0$,  the third moment theorem 
%\begin{align}\label{eq:dWmain}
%d_{W}(\rho_r,\gamma)
%  \leq 	\sqrt{2}m_{3}[\rho_r]
%\end{align}
%holds. 

We have  that  \begin{align}\label{eq:dWmainzero}
d_{W}(\rho_r,\gamma)
  \leq 	\frac{\sqrt{2}}{\sigma_r^3}\int_{\R^d}\E\left[\left|\int_0^r\langle\mu^x_s,\psi\rangle ds\right|^3\right] \nu(dx).
\end{align}
Moreover, under the condition $\psi\geq 0$,  the following third moment theorem holds
\begin{align}\label{eq:dWmain}
d_{W}(\rho_r,\gamma)
  \leq 	\sqrt{2}m_{3}[\rho_r].
\end{align}
\end{theorem}
\begin{remark}
We can adapt the arguments of the proof of Theorem \ref{thm:mainabstractone}, in order to show that  under the general condition of $\psi$ not being necessarily positive and assuming the finiteness of the moment of order four of $\rho_r$, it holds that  
\begin{align}\label{eq:dzetamain}
d_{\zeta}(\rho_r,\gamma)
  \leq 	m_{3}[\rho_r]+\frac{1}{6}(m_{4}[\rho_r]-3),
\end{align}
where the distance in the left is defined as
\begin{align*}
d_{\zeta}(\beta_1,\beta_2)
  &:=\sup_{f\in {\mathcal{K}}}|\langle\beta_1,f\rangle-\langle\beta_2,f\rangle|,
\end{align*}
where $\mathcal{K}:=\{f:\R\rightarrow\R\ ;\ |f^{\prime\prime}|,|f^{\prime}|\leq 1\}$.
\end{remark}

\begin{remark}
We emphasize the generality of the motion of the underlying measures: no structural assumptions are imposed on the dynamics of the process. Naturally, the main challenge lies in obtaining sharp estimates for the third and fourth moments of the model. In this context, the inequality \eqref{eq:dWmainzero}  proves to be a particularly effective tool.
\end{remark}

In Section~\ref{sec:particle:systems}, we demonstrate the applicability of inequality \eqref{eq:dWmainzero} in the context of the Dirac particle systems introduced in Section~\ref{sec:intro}, implementing it across a broad range of dynamics for the underlying process $\mathbf{X}$, including $\alpha$-stable processes, fractional Brownian motion, and solutions to elliptic stochastic differential equations driven by a Brownian motion. In Section~\ref{sec:spectralempirical}, we  present  implementations in the framework of the Dyson Brownian motion.

\begin{proof}[Proof of Theorem \ref{thm:mainabstractone}]
By \eqref{eq:Steinineq}, it suffices to bound 
$$\varepsilon_{r}:=|\E[Z_r[\psi]f(Z_r[\psi])-f^{\prime}(Z_r[\psi])]|,$$ 
for $f$ in $\mathcal{C}$. Being a random variable in the probability space $\tilde{\Omega}$, we can regard it as a function of an element $\eta$ in $\mathbf{N}_{\sigma}$ and a point $\omega$ in $\Omega$; which we denote as $Z_r(\eta,\omega)$, omitting the dependence on $\psi$ for convenience. Recall from Section \ref{sec:intro} that $\vartheta$
 denotes the distribution of the Poisson point process component. The quantity of interest can then be written as 
 \begin{align*}
\varepsilon_{r}
    &=\int_{\Omega}\varepsilon_{r}^{N_{\sigma}}(\omega)\Q(d\omega),
\end{align*}
 where 
  \begin{align*}
\varepsilon_{r}^{N_{\sigma}}(\omega)
  &:=\left|\int_{\mathbf{N}_{\sigma}}Z_{r}(\eta,\omega)f(Z_{r}(\eta,\omega))-f^{\prime}(Z_{r}(\eta,\omega))\vartheta(d\eta)\right|.
 \end{align*}
From \eqref{mainZ}  and some elementary computations, for $\omega\in\Omega$, we have 
\begin{multline*}
\int_{\mathbf{N}_{\sigma}}Z_{r}(\eta,\omega)f(Z_{r}(\eta,\omega))\vartheta(d\eta)\\
  = \frac{1}{\sigma_r}\int_0^r\int_{\mathbf{N}_{\sigma}}\int_{\mathbb{X}}\int_{\mathbb{Y}} \psi(y)(f-\E[f(Z_r)])(Z_{r}(\eta,\omega))\mu_t^x(\omega,dy)\eta(dx)\vartheta(d\eta)dt,
\end{multline*}
so that by Mecke's equation \eqref{eq:Mecke}, we deduce that 
\begin{multline*}
\int_{\mathbf{N}_{\sigma}}Z_{r}(\eta,\omega)f(Z_{r}(\eta,\omega))\vartheta(d\eta)\\
\begin{aligned}
  &= \frac{1}{\sigma_r}\int_0^r\int_{\mathbf{N}_{\sigma}}\int_{\mathbb{X}}\int_{\mathbb{Y}} \psi(y)(f-\E[f(Z_r)])(Z_{r}(\eta+\delta_x,\omega))\mu_t^x(\omega,dy)\nu(dx)\vartheta(d\eta)dt.
\end{aligned}
\end{multline*}
Now, integrating with respect to $\mathbb{Q}$ over ${\Omega}$, we get 
\begin{align*}
\E[Z_{r}f(Z_{r})]
  &= \frac{1}{\sigma_r}\int_0^r\int_{\mathbf{N}_{\sigma}}\int_{\mathbb{X}}\Theta_t(x,\eta)\nu(dx)\vartheta(\eta)dt,
\end{align*}
where 
\begin{align*}
\Theta_t(x,\eta)
  &:=\int_{{\Omega}}\int_{\mathbb{Y}} \psi(y)f(Z_{r}(\eta,\omega)+u_r^{\psi}(x,\omega)/\sigma_r)\mu_t^x(\omega,dy)\Q(d \omega)\\
  &-\int_{{\Omega}}\int_{\mathbb{Y}} \psi(y) \E[f(Z_r)]\mu_t^x(\omega,dy)\Q(d \omega),
\end{align*}
 with 
 \begin{align}\label{add-func-one-perticle}
     u_r^\psi(x,\omega):=\int_0^r\langle\mu^x_t(\omega),\psi\rangle dt.
 \end{align}
 Next we show that for every $x$ in $\mathbb{X},$ there exists a set $\mathcal{J}_x$ in $\mathbf{N}_{\sigma}(\mathbb{X})$, such that $\vartheta[\mathcal{J}_x]=1$, satisfying that for all $\eta$ in $\mathcal{J}_x$, the random variable $Z_r(\eta,\omega)$ and the process  $\{\langle \mu_t^x,\psi\rangle\ t\geq 0\}$, regarded as a functions of $\omega$, are stochastically independent with respect to $ \mathbb{Q}$. To verify this, due to the boundedness of the random variables $\langle \mu_t^x,\psi\rangle$ and $Z_r(\eta,\cdot )$, it suffices to determine the behavior of the corresponding mixed moments for an arbitrary finite choice of $t$'s. More precisely, if $t_1,\dots, t_{m}$ are non-negative real numbers and $\alpha_1,\dots, \alpha_m,\beta$ are positive integers, then we are required to show 
\begin{align}\label{eq:Tidentitygoal}
T
  &:=\int_{\mathbf{N}_{\sigma}(\mathbb{X})}\int_{\Omega}\langle \mu_{t_1}^x(\omega),\psi\rangle^{\alpha_1}\cdots \langle \mu_{t_m}^x(\omega),\psi\rangle^{\alpha_m}(\sigma_rZ_r(\eta,\omega))^{\beta}\mathbb{Q}(d\omega)\vartheta(d\eta)\nonumber\\
    &=\left(\int_{\mathbf{N}_{\sigma}(\mathbb{X}}\int_{\Omega}\langle \mu_{t_1}^x(\omega),\psi\rangle^{\alpha_1}\cdots \langle \mu_{t_m}^x(\omega),\psi\rangle^{\alpha_m}\mathbb{Q}(d\omega)\vartheta(d\eta)\right)\nonumber\\
    &\times \left(\int_{\mathbf{N}_{\sigma}(\mathbb{X})}\int_{\Omega}(\sigma_rZ_r(\eta,\omega))^{\beta}\mathbb{Q}(d\omega)\vartheta(d\eta)\right).
\end{align}
To this end, we first define $\mathcal{J}_x$ as the event $\{\eta\in\mathbf{N}_{\sigma}(\mathbb{X})\ ;\ x\notin\text{supp}[\eta]\}$. Since for every $A$ in $\mathcal{X}$, it holds that $\eta\mapsto \eta[A]$ is a Poisson random variable under $\vartheta$ with parameter $\nu[A]$, and since $\nu$ is non-atomic, we conclude that $\eta[\{x\}]$ is equal to zero $\vartheta$-almost everywhere, thus implying the existence of a set $\mathcal{J}_x$, such that $\{\eta\in\mathbf{N}_{\sigma}(\mathbb{X})\ ;\ x\notin\text{supp}[\eta]\}$ for  $\eta$ in $\mathcal{J}_x$. Next, we observe that for $\eta$ in $\mathcal{J}_x$, it holds that $\mu^x$ is independent of $\{\mu^y\ ;\ y\in\text{supp}[\eta]\}$, and consequently due to the condition of independence over the $\mu^z$'s. We thus conclude that  

\begin{multline*}
\int_{\Omega}\langle \mu_{t_1}^x(\omega),\psi\rangle^{\alpha_1}\cdots \langle \mu_{t_m}^x(\omega),\psi\rangle^{\alpha_m}
    \langle \mu_{s_1}^{y_1}(\omega),\psi\rangle\cdots \langle \mu_{s_{\beta}}^{y_{\beta}}(\omega),\psi\rangle \mathbb{Q}(d\omega)\\
\begin{aligned}
  &=\left(\int_{\Omega}\langle \mu_{t_1}^x(\omega),\psi\rangle^{\alpha_1}\cdots \langle \mu_{t_m}^x(\omega),\psi\rangle^{\alpha_m}\mathbb{Q}(d\omega)\right)
    \left(\int_{\Omega}\langle \mu_{s_1}^{y_1}(\omega),\psi\rangle\cdots \langle \mu_{s_{\beta}}^{y_{\beta}}(\omega),\psi\rangle \mathbb{Q}(d\omega)\right).
\end{aligned}
\end{multline*}
Relation \eqref{eq:Tidentitygoal} follows from here. Hence we conclude that $Z_r(\eta,\omega)$ and the process  $\{\langle \mu_t^x,\psi\rangle\ t\geq 0\}$ are independent $\vartheta$-almost everywhere, as required.\\

\noindent This observation allows us to write the following simplified version of $\Theta_t$, 
\begin{align*}
\Theta_t(x,\eta)
  &=\int_{{\Omega^2}}\int_{\mathbb{Y}} \psi(y)(f(Z_{r}(\eta,\omega)+u_r^{\psi}(x,v)/\sigma_r)- \E[f(Z_r)])\mu_t^x(v,dy)\Q(dv)\Q(d \omega).
\end{align*}
From here it follows that 
\begin{align*}
\E[Z_{r}f(Z_{r})]
  &= \frac{1}{\sigma_r}\int_0^r\int_{\mathbb{X}}
  \int_{{\Omega}}\int_{\mathbb{Y}} \psi(y)\E[f(Z_{r}+u_r^{\psi}(x,v)/\sigma_r)- f(Z_r)]\mu_t^x(v,dy)\Q(dv)\nu(dx)dt\\
  &= \frac{1}{\sigma_r}\int_{\mathbb{X}}
  \int_{{\Omega}} \E[f(Z_{r}+u_r^{\psi}(x,v)/\sigma_r)- f(Z_r)]u_r^{\psi}(x,v)\Q(dv)\nu(dx).
\end{align*}
Adding and subtracting a linear approximation for $f(Z_{r})$ in the right-hand side, and then taking expectation, we obtain
\begin{align}\label{eq:ZrwithfZrfirst}
\E[Z_rf(Z_r)]
  &=\frac{\E[f^{\prime}(Z_{r})]}{\sigma_r^2}\int_{\mathbb{X}}
  \int_{{\Omega}} |u_r^{\psi}(x,v)|^2\Q(dv)\nu(dx)+\mathcal{E}_r[f],
\end{align}
where 
\begin{align}\label{Erdef}
\mathcal{E}_r[f]
  &:=\frac{1}{\sigma_r}\int_{\mathbb{X}}
  \int_{{\Omega}} \E[f(Z_{r}+u_r^{\psi}(x,v)/\sigma_r)- f(Z_r)- f^{\prime}(Z_r)(u_r^{\psi}(x,v)/\sigma_r)]u_r^{\psi}(x,v)\Q(dv)\nu(dx).
\end{align}
One can easily check that in the case $f(x)=x$, $f^{\prime}(x)=1$ the term $\mathcal{E}_r[f]=0$ and $$\E[Z_rf(Z_r)]=\E[|Z_r|^2]=Var[Z_r]=1,$$ 
and consequently, by \eqref{eq:ZrwithfZrfirst} and the fact that $Z_r$ is a centered random variable, 
\begin{align}\label{eq:identitysigmar}
1
  &=\frac{1}{\sigma_r^2}\int_{\mathbb{X}}
  \int_{{\Omega}} |u_r^{\psi}(x,v)|^2\Q(dv)\nu(dx).
\end{align}
We conclude that 
\begin{align}\label{eq:ZrwithfZrsecond}
\E[Z_rf(Z_r)]
  &=\E[f^{\prime}(Z_{r})]+\mathcal{E}_r[f],
\end{align}
which yields $\varepsilon_r=\mathcal{E}_r[f]$.  The Taylor approximation applied to \eqref{Erdef} yields 
\begin{align*}
\mathcal{E}_r[f]
  &\leq \frac{\|f^{\prime\prime}\|}{\sigma_r^3}\int_{\mathbb{X}}
  \E[|u_r^{\psi}(x,v)|^3]\nu(dx).
\end{align*}
Since $f$ belongs to $\mathcal{C}$, we the following 
\begin{align*}
\mathcal{E}_r[f]
  &\leq \frac{2}{\sigma_r^3}\int_{\mathbb{X}}
  \E[|u_r^{\psi}(x,v)|^3]\nu(dx),
\end{align*}
which finishes the proof of \eqref{eq:dWmainzero}.\\

\noindent Now we proceed with the proof of \eqref{eq:dWmain}. To this end, we set $M_{\alpha}(x):=|x|^{\alpha}$, with $\alpha$ in $\R$.  Observe that taking $f(x)=M_2(x)$ in \eqref{eq:ZrwithfZrsecond} yields the identity
\begin{align}\label{eq:secondmain}
\E[Z_r^3]
  &=\E[Z_r]+\mathcal{E}_r[M_{2}]=\mathcal{E}_r[M_{2}].
\end{align}
From identity \eqref{Erdef}, we can easily deduce that 
\begin{align}\label{eq:secondmainprime}
\mathcal{E}_r[M_{2}]
  &:=\frac{1}{\sigma_r^3}\int_{\mathbb{X}}
  \int_{{\Omega}} u_r^{\psi}(x,v)^3\Q(dv)\nu(dx).
\end{align}
Relation \eqref{eq:dWmain} follows from \eqref{eq:dWmainzero} \eqref{eq:secondmain} and \eqref{eq:secondmainprime}

\end{proof}

\section{Application to Particle Systems}\label{sec:particle:systems}

\noindent In this section, we apply Theorem \ref{thm:mainabstractone} to the case where the measures $\mu_t^\mathbf{x}$ are given by
\[
\mu_t^{\mathbf{x}} := a_1 \delta_{X_t^{1,\mathbf{x}}} + \cdots + a_m \delta_{X_t^{m,\mathbf{x}}},
\]
where $\mathbf{X}_t^{\mathbf{x}} = (X_t^{1,\mathbf{x}}, \dots, X_t^{m,\mathbf{x}})$ is a sequence of $(\R^{d})^{m}$-valued processes satisfying 
\begin{align*}
X_{t}^{i,\mathbf{x}}
  &=X_{t}^{i,0}+\mathbf{x},	
\end{align*}
for a given choice of $m$ in $\N$. To state the main applications precisely, we now introduce some notation and hypotheses that will be used throughout this section. The random vector $(\mathbf{X}_{t_1}^{i_1},\mathbf{X}_{t_2}^{i_2} , \mathbf{X}_{t_3}^{i_3} )$ will be assumed to have a joint density $f_{t_1,t_2,  t_3}^{\mathbbm{i}}$, where $\mathbbm{i}$ is a multi-index of the form  $\mathbbm{i}=(i_1,i_2,i_3)$, with $1\leq i_1,i_2,i_3\leq m$. \\

\noindent Particle systems with random initial configurations and independent motions have long been recognized as a natural framework to study fluctuation phenomena via occupation times. Classical works on Poisson systems of independent Brownian particles and critical branching Brownian motion  show that suitably rescaled occupation time functionals exhibit non-trivial Gaussian limits and long-range dependence, with phase transitions in the spatial dimension and in the branching mechanism \cite{C-G,D-W,C-G1984}. In particular, for critical branching systems with symmetric $\alpha$-stable motion, the occupation time fluctuations converge, in appropriate regimes, to self-similar Gaussian processes such as fractional and sub-fractional Brownian motions, or to generalized Wiener processes taking values in the set of tempered distributions, depending on whether the dimension is subcritical, critical, or supercritical \cite{BGT-I,BGT-II}. More recent results extend this picture to age-dependent branching and other variants, confirming the robustness of occupation-time based limits \cite{LM-MS-RG}. At the same time, the measure-valued limits of such systems, and their scaling behavior, are tightly connected with the general theory of superprocesses and measure-valued diffusions \cite{Dawson1993,Etheridge2000,Perkins2002,Dynkin1991}, and with the broader literature on long-range dependence and self-similar processes \cite{MR1280932}. In this sense, Poisson-driven particle systems provide not only a rich source of non-trivial Gaussian limits, but also a unifying microscopic model for a wide class of fluctuation phenomena that our abstract results are designed to capture.\\

\noindent In order to encompass virtually all of the particle-system scenarios discussed above, we introduce the technical condition \textbf{(H)} below on the dynamics of the individual particles. As shown later, this condition is satisfied not only for the classical cases of $\alpha$-stable motions (traditionally the main setting in the occupation-time literature), but also for the largely unexplored regimes of uniformly elliptic Markov diffusions and long-range dependent motions, including particles evolving as fractional Brownian paths. It is worth emphasizing that this includes situations where the Markov property fails or where the trajectories are not even semimartingales, settings for which the fluctuation theory is far from complete (and, in several cases, almost entirely unexplored). Our framework thus extends the existing literature in two directions: (i) it allows substantially more general initial configurations than homogeneous Poisson fields for all the models under consideration (including the classical ones), and (ii) it accommodates a wide spectrum of particle dynamics, ranging from classical Markovian motions and heavy-tailed processes to genuinely non-Markovian models with memory. The anticipated technical condition is presented next\\

\noindent\textbf{(H)} There are strictly positive constants $\kappa$ and $\beta$, such that for all  $\mathbf{a},\mathbf{b},\mathbf{c}\in(\R^{d})^m$ and nonnegative numbers $s_1,s_2$, satisfying  $s_1,s_2<s_3$ and $\mathbbm{i}$ as above, the conditional density of the random vector $(\mathbf{X}_{s_1}^{i_1},\mathbf{X}_{s_2}^{i_2},\mathbf{X}_{s_3}^{i_3})$
$$f_{s_3|s_1,s_2}^{\mathbbm{i}}(\mathbf{a} | (\mathbf{b},\mathbf{c})):=f_{s_1,s_2,s_3}^{\mathbbm{i}}(\mathbf{b},\mathbf{c},\mathbf{a})/f_{s_1,s_2}^{(i_1,i_2)}(\mathbf{b},\mathbf{c}),$$
satisfies 
\begin{align*}
\sup_{\mathbf{a},\mathbf{b},\mathbf{c}\in\R^{md}}f_{s_3|s_1,s_2}^{\mathbbm{i}}(\mathbf{a} | (\mathbf{b},\mathbf{c}))	&\leq \kappa |s_3-(s_1\vee s_2)|^{-\beta d}.
\end{align*}
%{\color{cyan}propongo que lo rojo reemplace esto
%$$f_{s_1,s_2,s_3}^{\mathbbm{i}}(\mathbf{a},\mathbf{b},\mathbf{c})/f_{s_1,s_2}^{(i_2,i_3)}(\mathbf{b},\mathbf{c}),$$  denoted by %$f_{s_3|s_1,s_2}^{\mathbbm{i}}(\mathbf{a} | (\mathbf{b},\mathbf{c}))$, satisfies 
%\begin{align*}
%\sup_{\mathbf{a},\mathbf{b},\mathbf{c}\in\R^{md}}f_{s_3|s_1,s_2}^{\mathbbm{i}}(\mathbf{a} | (\mathbf{b},\mathbf{c}))	&\leq \kappa |s_3-(s_1\vee %s_2)|^{-\beta d}.
%\end{align*}
%}
The next theorem provides a key upper bound to handle the left-hand side of (\ref{eq:dWmainzero}).

\begin{theorem}\label{thm:application}
If $\psi$ is a non-negative function, then under condition \textbf{(H)}, the distribution $\rho_r$ of $Z_r$ satisfies 
\begin{align}\label{eq:dWmainzero-1}
d_{W}(\rho_r,\gamma)
  \leq 	K\frac{r^{1-\beta d}}{\sigma_r},
\end{align} 
for some positive constant $K$ independent of $r$.
\end{theorem}
\noindent The proof of the above result, presented in Section \ref{sec:secondhtoemre}, builds on Theorem \ref{thm:mainabstractone} and makes essential use of the process density, which is controlled by condition \textbf{(H)}. Although these density conditions could possibly be relaxed, identifying the sharpest assumptions appears to be a separate problem beyond the scope of this paper. We defer it to future research.\\
 
\noindent In the following subsections, we present several particular instances of the above result, including well known cases in the literature, which are typically established only for Poisson configurations with Lebesgue control measure. This assumption is substantially relaxed in the present paper. It is also worth emphasizing that our approach is fundamentally different from classical methods. Traditionally, such results are obtained by exploiting the Markovian structure of the underlying dynamics, deriving evolution equations for the relevant moments, and then analyzing these equations to identify the limiting behavior. In contrast, our method, based on the Stein framework introduced in Section \ref{sec:Stein-intro}, provides a perspective that considerably simplifies the proof mechanism.

\subsection{Particles performing $\alpha$-stable motions}\label{sec:firstpresentslpha}
In this subsection, we study the specific case in which the processes $\mu^x$, defined in~\eqref{eq:mutdefsumdeltas}, are such that the vector $\mathbf{X}^x$ follows an $\alpha$-stable process, with $0<\alpha\leq 2$. Before addressing the problem in detail, we provide a brief overview of related previous work.  \\

\noindent Our starting point is the work of~\cite{BGT-I,BGT-II}, where the authors study a Poissonian system with Lebesgue control measure, consisting of particles that generate independent $\alpha$-stable motions on $\mathbb{R}^d$. In their setting, the authors consider an independent collection of $d$-dimensional processes $\{\mathbf{X}^x : x \in \mathbb{R}^d\}$, each with independent $\alpha$-stable components started at the origin. This family naturally induces the random measure  
\[
\mu_t^{x}(dy) := \delta_{\,x + \mathbf{X}_t^0}(dy).
\]
Then, for a fixed test function $\psi:\R^{d}\rightarrow\R$, define the normalized additive functional
\begin{align}\label{mainZ-T}
\mathcal{Z}_T[\psi] := \frac{1}{F_T} \left(A_T[\psi] - \mathbb{E}[A_T[\psi]]\right),
\end{align}
where the scaling factor $F_T$ is defined as
\begin{align*}
F_T :=
\begin{cases}  
T^{1-d/2\alpha} & \text{if } d < \alpha, \\
(T/\log T)^{-1/2} & \text{if } d = \alpha, \\
T^{-1/2} & \text{if } d > \alpha.
\end{cases}
\end{align*}
It is shown in~\cite[Theorem 2.1]{BGT-I} and~\cite[Theorem 2.1]{BGT-II} that $\mathcal{Z}_T$ converges in distribution, as $T \to \infty$, to a normal random variable $\mathcal{N}(0, \sigma^2)$, for a suitably chosen variance $\sigma^2$ that depends on the parameters of the system but not on the number of particles. Related results were obtained earlier in~\cite{D-W} for systems in which the particles follow independent Brownian motions in $\mathbb{R}^d$, corresponding to the case $\alpha=2$. These developments build on the seminal work~\cite{C-G1984}, where analogous results were established for Poissonian systems of particles performing independent random walks on the lattice $\mathbb{Z}^d$.

\medskip

\noindent It is worth to recall that all of the aforementioned results are focused exclusively on establishing distributional convergences (in particular, functional convergence) under various parameter settings for the model. However, to our knowledge, there are no results in the literature that discuss the rate of convergence of the distribution of $\mathcal{Z}_t$  as $t$ tends to infinity. The upcoming corollary, derived from Theorem \ref{thm:application}, is the first to address this specific question.

\medskip
\noindent We now return to the more general framework of Dirac particle systems introduced in Section \ref{sec:modeldesc}, which includes as particular cases the models discussed above. Specifically, we consider $\alpha$-stable processes $\mathbf{X}_t^x$ composed of $m$ independent blocks,  
\[
\mathbf{X}_t^{x} = \big(\mathbf{X}_t^{x,1}, \dots, \mathbf{X}_t^{x,m}\big),
\] 
where each $\mathbf{X}_t^{x,i}$ is a symmetric $\alpha$-stable process in $\mathbb{R}^d$. In this setting, we observe that $\alpha \beta = 1$. Consequently, the condition $1>\beta d$ is equivalent to requiring $d$ is equal to one. Under the assumptions of Theorem \ref{thm:application}, and in view of \eqref{sec:firstpresentslpha}, an application of Theorem \ref{thm:application} yields the following result. Throughout this section, we shall assume that the above mentioned conditions on the process are satisfied. One can easily check that condition \textbf{(H)} holds for $\beta=1/\alpha$. This observation, in combination with Lemma \ref{eq:sigmafrombelowstable} in Appendix A, yields the following result. 

\begin{Proposition}\label{eq:propalphastablecaseapplication}
If $d=1$ and $\psi\geq0$, then the Wasserstein distance between $\rho_r$ and the standard Gaussian distribution $\gamma$ satisfies
\[
d_W(\rho_r,\gamma) \;\leq\; Cr^{-1/(2\alpha)}.
\]
for some constant $C>0.$
\end{Proposition}

\subsection{Particles moving as a diffusion on $\R^d$}\label{eq:sec:diffusivecaseexamples}

Let $a:\R^d\rightarrow \R^{d\times d}$ be a measurable symmetric function(matrix) satisfying the uniform ellipticity condition
\[
\eta I_d \leq a(x) \leq \frac{1}{\eta} I_d,
\]
for $x$ in $\R^d$, for some $0<\eta\leq 1$, where $I_d$ denotes the $d$-dimensional identity matrix. We additionally assume that $a$ has bounded derivatives of all orders.  Under these conditions, consider the diffusion process $X=\{X_t:t\ge 0\}$ on $\R^d$ associated to the operator
\[
\mathcal{L}=\nabla\cdot (a\nabla),
\]
and let $P=\{P_t:t\ge 0\}$ denote its associated Feller semigroup. Let $\psi$ be a given test function in $C_b(\R^d)$. It is well known that the semigroup $P$ admits a transition density $p(t,x,y)$ such that
\[
P_t\psi(x)=\int_{\R^d}\psi(y)\,p(t,x,y)\,dy.
\]
Moreover, under our assumptions on $a$, the density $p$ is infinitely differentiable in the interior of any compact subset of $(0,\infty)\times\R^d\times\R^d$ and solves the Fokker-Planck equation
\[
\frac{\partial}{\partial t}p(t,x,y)
   = \mathcal{L}p(t,x,\cdot)(y),\,\, \mbox{for} \,\, (t,x,y)\in(0,\infty)\times\R^d\times\R^d. 
\]
An important feature of such uniformly elliptic diffusions is that their transition density exhibits Gaussian-type upper and lower bounds. In particular, $p(t,x,y)$ is symmetric in $(x,y)$ and satisfies
\begin{align}\label{bounded-density-diffusion}
    \frac{1}{K t^{d/2}}
      \exp\!\left(-\frac{K|y-x|^2}{t}\right)
    \;\leq\;
    p(t,x,y)
    \;\leq\;
    \frac{K}{t^{d/2}}
      \exp\!\left(-\frac{|y-x|^2}{K t}\right),
\end{align}
where $K$ is some constant larger than or equal to one, only depending on $\eta$ and $d$. The reader is referred to (I.0.10) in \cite{Stroock} for further details about \eqref{bounded-density-diffusion}.\\
% {\color{red} Hay que obtener tambi\'en la cota por abajo. Luego se va a ocupar}.\\

\noindent Now,  let us consider a Poissonian  particle system such that the particles follow independent motions on $\R^d$ driven by the diffusion process $X$ described above. One can easily check that condition \textbf{(H)} holds for $\beta=1/\alpha$, $\alpha=2$. This observation, in combination with Lemma \ref{eq:sigmafrombelowdiffusion} in Appendix 1, yields the following result.

\begin{Proposition}\label{eq:propdiffusioncaseapplication}
If $d=1$ and $\psi\geq0$, then the Wasserstein distance between $\rho_r$ and the standard Gaussian distribution $\gamma$ satisfies
\[
d_W(\rho_r,\gamma) \;\leq\; Cr^{-1/4}.
\]
for some constant $C>0.$
\end{Proposition}

\noindent It is important to stress that, to the best of our knowledge, the literature contains no qualitative fluctuation results for particle systems driven by general elliptic Markov diffusions. Our result goes beyond this gap by simultaneously establishing the qualitative behavior and providing an explicit rate of convergence.

\begin{remark}\label{remark:gap_in_literature}
There is a clear gap in the literature regarding a functional version of the result above.  
While functional limit theorems for occupation-time fluctuations have been established in a variety of classical settings (most notably for branching and non-branching systems driven by Brownian or $\alpha$-stable motions \cite{C-G,BGT-I,BGT-II,LM-MS-RG}), no analogous functional convergence results appear to be available for particle systems governed by general elliptic diffusions or their non-Markovian counterparts to be presented in Section \ref{eq:fractionalbrownianmotionparticles}.  Exploring this functional behavior remains an interesting direction for future research.
\end{remark}

\subsection{Particles moving as a fractional Brownian motion on $\R^d$}
\label{eq:fractionalbrownianmotionparticles}

We now turn to the case where the particle motions are driven by fractional Brownian paths. A fractional Brownian motion $B^{H}=\{B_{t}^{H},\, t\ge 0\}$ in $\R^{d}$ with Hurst parameter $H$ in $(0,1)$, is a centered Gaussian process with covariance
\[
\mathbb{E}\!\left[ B_t^H \cdot B_s^H \right]
   = \tfrac12 \left( t^{2H} + s^{2H} - |t-s|^{2H} \right) I_d.
\]
Fractional Brownian motion plays a distinguished role among Gaussian processes: it is self-similar, has stationary increments, and its Hurst parameter acts as a fundamental tuning parameter for both path regularity and the presence (or absence) of long-range dependence. These features make $B^{H}$ a canonical model for highly oscillatory phenomena and time series with memory effects, where neither the Markov property nor the semimartingale structure is appropriate. We refer the reader to the monographs of Mishura \cite{Mishura2008} and Nourdin \cite{NourdinFBM} for further background on fractional Brownian motion and its applications.\\

Within the framework of the present paper, particle systems driven by $B^{H}$ thus provide a natural venue for initiating the study of temporal correlations in the motion of the evolving measures. Consider now a Poissonian particle system such that the particles follow independent motions on $\R^d$ driven by $B^H$. One can check that condition \textbf{(H)} holds with $\beta = H$ due to the local non-determinism property of the fractional Brownian motion (see \cite{Xiao2009SLND}). To verify the validity of \textbf{(H)}, we observe that for every $s_1,s_2< s_3$, the random vector $(\mathbf{X}_{s_1}^{i_1},\mathbf{X}_{s_2}^{i_2},\mathbf{X}_{s_3}^{i_3})$
%\[
%\big(\mathbf{X}_{s_1}^{i_1},\mathbf{X}_{s_2}^{i_2},\mathbf{X}_{s_3}^{i_3}\big)
%\]
is Gaussian. Thus, if we denote its conditional mean and conditional covariance by
\[
m_{s_3|s_1,s_2}^{\mathbbm{i}}(\mathbf{b},\mathbf{c})
   := \mathbb{E}\big[\mathbf{X}_{s_3}^{i_3}\,\big|\,\mathbf{X}_{s_1}^{i_1}=\mathbf{b},
   \mathbf{X}_{s_2}^{i_2} =\mathbf{c}],
\]
and
\[
v_{s_3|s_1,s_2}
   := \mathrm{Var}\big[ B_{s_3}^H \,\big|\,
        B_{s_1}^H,B_{s_2}^H \big],
\]
respectively. The conditional density can be written as
\begin{align}\label{eq:densitycondgauss}
f_{s_3|s_1,s_2}^{\mathbbm{i}}(\mathbf{a}\mid(\mathbf{b},\mathbf{c}))
=
\frac{1}{(2\pi v_{s_3|s_1,s_2})^{md/2}}
\exp\!\left(
  -\frac{1}{2v_{s_3|s_1,s_2}}
  \big\|\mathbf{a} - m_{s_3|s_1,s_2}^{\mathbbm{i}}(\mathbf{b},\mathbf{c})\big\|^2
\right),
\end{align}
By the strong local nondeterminism property of fractional Brownian motion (see \cite{Xiao2009SLND}), there exists a constant $c>0$ such that
\begin{align}\label{eq:variancebound}
v_{s_3|s_1,s_2}
   := \Var\!\big( B_{s_3}^H \,\big|\,
        B_{s_1}^H, B_{s_2}^H \big)
   \;\ge\;
   c\,\big|s_3 - (s_1 \vee s_2)\big|^{2H}.	
\end{align}
Condition \textbf{(H)} follows from \eqref{eq:densitycondgauss} and \eqref{eq:variancebound}. This observation, in combination with Lemma \ref{eq:sigmafrombelowfBm} in the Appendix A, yields the following result.

\begin{Proposition}\label{eq:propfractionalcaseapplication}
The Wasserstein distance between $\rho_r$ and the standard Gaussian distribution $\gamma$ satisfies
\[
d_W(\rho_r,\gamma) \;\leq\; C (r^{-3H/4}\Indi{\{d=1\}}
+r^{-3H/2} \log(r)^{-1/2}\Indi{\{d=2\}} 
+ r^{-H(2d-1)/2}\Indi{\{d\geq 3\}}).
\]
for some constant $C>0.$
\end{Proposition}

\noindent
The proof of Proposition~\ref{eq:propfractionalcaseapplication} relies crucially on establishing a suitable lower bound for the conditional variances in condition \textbf{(H)}. The argument ultimately proceeds by comparing the fractional Brownian motion model with an adjustment of the standard Brownian motion. We do not claim that the resulting rates are sharp, and improving them (or even identifying the optimal dependence on $H$ and $d$ for the validity of the asymptotic gaussianity), is left as an interesting question for future research. In the same vein, the problem of establishing a  {functional} version of the convergence in Proposition \ref{eq:propfractionalcaseapplication} remains completely open.

\section{Application to the Dyson process}\label{sec:spectralempirical}
In this section, we implement our methodology in the setting of the Dyson Brownian motion. This process, first introduced by Dyson in his seminal 1962 work \cite{Dyson1962}, provides a canonical stochastic model for the dynamics of eigenvalues of large symmetric matrices. The reader is referred to \cite{AndersonGuionnetZeitouni2010,Mehta2004} for further references on the topic. The details of the model are as follows: let $W_t$ denote a real symmetric matrix-valued Brownian motion with independent Brownian entries. The Dyson Brownian motion can be defined as the matrix process
\[
X_t \;=\; X_0 \;+\; \frac{1}{\sqrt{d}}\, W_t,
\]
where $X_0$ is an initial real symmetric matrix. A fundamental feature of $X_t$ is that its eigenvalue dynamics can be described explicitly. Let $\lambda_1(t),\dots,\lambda_d(t)$ denote the eigenvalues of $X_t$. By a careful application of It\^o calculus techniques, one can show that the eigenvalues satisfy the celebrated Dyson system
\begin{align}\label{eq:dynamicslambda}
d\lambda_i(t) \;=\; \frac{1}{\sqrt{d}}\, dB_i(t) 
+ \frac{1}{2d}\sum_{j\ne i}\frac{1}{\lambda_i(t)-\lambda_j(t)}\,dt,	
\end{align}
where $B_1,\dots,B_d$ are independent standard real-valued Brownian motions. The evolution depends on the initial matrix $X_0$ only through its spectrum, so the dynamics is fully determined by the initial configuration $\{\lambda_1(0),\dots,\lambda_d(0)\}$, regardless of the choice of eigenvectors. It is far from obvious that the evolution \eqref{eq:dynamicslambda} is well defined, since an eigenvalue collision would cause the interaction terms on the right-hand side to blow up, compromising the well-posedness of the system. Fortunately, it is a well-known result that such collisions almost surely do not occur (Th. 4.3.2 in \cite{AndersonGuionnetZeitouni2010}).  Once the well-posedness of the system is guaranteed, the process  
\[
\lambda(t) \;=\; \big(\lambda_1(t),\dots,\lambda_d(t)\big)
\]  
becomes a Markov diffusion on $\mathsf W_d:=\{\lambda\in\mathbb R^d:\ \lambda_1<\cdots<\lambda_d\}$ with infinitesimal generator  
\[
\mathcal{L} f(\lambda) \;=\; \frac{1}{2d}\sum_{i=1}^d \partial_{i}^2 f(\lambda)
\;+\; \frac{1}{2d}\sum_{i\neq j}\frac{1}{\lambda_i-\lambda_j}\,\partial_i f(\lambda),
\]
for smooth test functions $f:\R^d \rightarrow \R$. In the sequel, $\upsilon_{t}$ will denote the transition kernel associated to the above infinitesimal generator, which is confined to the region of non-collision $\mathsf W_d$. To connect with the theory developed throughout the paper, we introduce the empirical point measure  
\[
\mu_t^d \;=\; \sum_{i=1}^d a_i\,\delta_{\lambda_i(t)},
\]
where $a_1,\dots,a_{d}>0$, and consider the  additive functionals associated defined earlier. As in the previous sections, this measure-valued dynamics, together with the initial Poissonian configuration introduced in the introduction, naturally induces a renormalized additive functional. To establish its asymptotic normality along with the corresponding rate of convergence, as presented in Proposition~\ref{eq:dysoncaseapplication} below, we follow the same strategy used in the proof of Theorem~\ref{thm:application}. Define 
\begin{align*}
\mathcal{R}_r
  &:=\frac{1}{\sigma_r^3}\int_{\mathsf W_d}\E\left[\left|\int_0^r\langle\mu^x_s,\psi\rangle ds\right|^3\right] \nu(dx).
\end{align*}
As before, we can show that 
\begin{align*}
\mathcal{R}_r
  &\leq A\sigma_r^{-3}\int_{\mathsf W_d}\int_{[0,r]^3}\sum_{1\leq i_1,i_2,i_3\leq m}\Indi{\{s_1\leq s_2\leq s_3\}}a_{i_1}a_{i_2}\E[\psi(\mathbf{X}_{s_1}^{\mathbf{y},i_1})\psi(\mathbf{X}_{s_2}^{\mathbf{y},i_2})\psi(\mathbf{X}_{s_3}^{\mathbf{y},i_3})]d\mathbf{s}\nu(d\mathbf{y}),
\end{align*}
which by the Markov property yields 
\begin{align*}
\mathcal{R}_r
  &\leq A\sigma_r^{-3}\int_{\mathsf W_d^{4}}\int_{[0,r]^3}\sum_{1\leq i_1,i_2,i_3\leq m}\Indi{\{s_1\leq s_2\leq s_3\}}a_{i_1}a_{i_2}\upsilon_{s_1}(\mathbf{y},\mathbf{z}_1)
  \upsilon_{s_2-s_1}(\mathbf{z}_1,\mathbf{z}_2)
  \upsilon_{s_3-s_2}(\mathbf{z}_2,\mathbf{z}_3)\\
  &\times \psi(\mathbf{z}_1)
  \psi(\mathbf{z}_2)\psi(\mathbf{z}_3)d\mathbf{s}d\mathbf{z}\nu(d\mathbf{y}).
\end{align*}
Next, we bound the term 
\begin{align*}
J:=\int_{\mathsf W_d}\psi(\mathbf{z}_3)\,\upsilon_{s_3-s_2}(\mathbf{z}_2,\mathbf{z}_3)\,d\mathbf{z}_3
  \leq \|\psi\|_{\infty}\,\mathbb{P}[|\mathbf{\lambda}(s_3-s_2)|\leq r\mid \lambda(0)=\mathbf{z}_2],
\end{align*}
where $r$ denotes the diameter of the support of $\psi$. To bound the probability term on the right-hand side, we consider a matrix realization of the eigenvalue process itself. 
Let $G_t$ denote a symmetric Brownian motion (matrix) with independent standard Gaussian entries, 
and let $X$ be any real symmetric matrix whose eigenvalues coincide with the initial configuration $\lambda(0)=\mathbf{z}_2\in\mathsf W_d$. Let us define
\[
H_t \;:=\; X + G_t,
\]
so that the eigenvalue process of $H:=\{H_t, t\geq0 \}$ coincides in law with the Dyson Brownian motion started from $\lambda(0)=\mathbf{z}_2$. In what follows, all matrix norms will be understood as the Euclidean (Hilbert-Schmidt) norm, given by 
$$\|A\|^2:=\sum_{i,j=1}^d a_{ij}^2,$$
for $A$ having components $ a_{i,j}$, for $1\leq i,j\leq d$. 
Observe that  \[\|H_t\|^2=tr\left((X+G_t)^2\right)=\sum_{i=1}^d |\lambda_i(X+G_t)|^2.\]
Thus, if $|\lambda(X+G_t)|\leq S$, for some $S>0$, we have that $\|H_t\|\leq \sqrt{d}S$. This inequality, in combination with the fact that  
$$\|H_t\|=\|G_t +X\| \geq \|G_t\|-\|X\|,$$ 
yields  $\{\|H_t\|\leq \sqrt{d} S\}\subset \{\|G_t\|\leq S+\|X\|\}$. Hence, 
\[\mathbb{P}\left[|\lambda(H_t)|\leq \sqrt{d} S\right]\leq \mathbb{P}\left[\|G_t\|\leq \sqrt{d}S+\|X\|\right],\]
which implies 
\[\sup_{|\lambda(X)|\leq R}\mathbb{P}\left[|\lambda(H_t)|\leq \sqrt{d} S\right]\leq \mathbb{P}\left[\|G_t\|\leq \sqrt{d}S+R\right],\]
for any $R>0$.  Taking $R=\|\mathbf{z}_2\|$ and $S$ equal to the diameter of the support of $\psi$, we obtain
\[
\mathbb{P}\big[|\lambda(s_3-s_2)|\le S\mid \lambda(0)=\mathbf{z}_2\big]
\;\le\;
\mathbb{P}\big[\|G_{s_3-s_2}\|\le S+\|\mathbf{z}_2\|\big].
\]
Finally, since $G_t$ has the same law as a Gaussian symmetric matrix with independent $\mathcal N(0,1)$ entries on and above the diagonal and dilated by $\sqrt{t}$, we have
\[
\|G_t\|\stackrel{d}{=}\sqrt{t}\,\chi_D,
\]
where $\chi_D$ denotes the Euclidean norm of a standard Gaussian vector in $\mathbb{R}^D$, with $D:=d(d+1)/2$.
Consequently,
\[
\mathbb{P}\big[|\lambda(s_3-s_2)|\le S\mid \lambda(0)=\mathbf{z}_2\big]
\;\le\;
\mathbb{P}\!\left[\chi_D \le \frac{S+\|\mathbf{z}_2\|}{\sqrt{s_3-s_2}}\right].
\]
The $\chi$-small-ball bound $\mathbb{P}[\chi_D\le y]\le C y^{\frac{D}{2}}\wedge1$ and the fact that $\|\mathbf{z}_2\|\le R_\psi$ for $\mathbf{z}_2$ in the support of $\psi$, we obtain
\[
J \;\le\; C\,(s_3-s_2)^{-\frac{D}{4}},
\]
for some possibly different $C>0$. The above analysis, in combination with the fact that $\upsilon_{s_3-s_2}(\mathbf{z}_2,\cdot)$ is a probability kernel, implies that the term 
\begin{align}\label{eq:intergals3}
\int_{\mathsf W_d}
  \upsilon_{s_3-s_2}(\mathbf{z}_2,\mathbf{z}_3)\psi(\mathbf{z}_3)d\mathbf{z}_3.
\end{align}
is bounded by both the constant equal to $\|\psi\|_{\infty}$, and by a constant multiple of $(s_3-s_2)^{-D/4}$. If $d\geq 2$, we have that $D\geq 3$, and consequently, the integral of $s_{3}$ over $[s_3,\infty)$ of the term \eqref{eq:intergals3} is finite constant independent of $s_2$. Hence, it follows that 
\begin{align*}
\mathcal{R}_r
  &\leq C\sigma_r^{-3}\int_{\mathsf W_d^{3}}\int_{[0,r]^2}\sum_{1\leq i_1,i_2\leq m}\Indi{\{s_1\leq s_2\}}a_{i_1}a_{i_2}\upsilon_{s_1}(\mathbf{y},\mathbf{z}_1)
  \upsilon_{s_2-s_1}(\mathbf{z}_1,\mathbf{z}_2)\\
  &\times \psi(\mathbf{z}_1)
  \psi(\mathbf{z}_2)d\mathbf{s}d\mathbf{z}\nu(d\mathbf{y})=C\sigma_{r}^{-1}.
\end{align*}
From here, the following result about the distribution $\rho_{r}$ of the variable \eqref{mainZ} holds

\begin{Proposition}\label{eq:dysoncaseapplication}
The Wasserstein distance between $\rho_r$ and the standard Gaussian distribution $\gamma$ satisfies
\[
d_W(\rho_r,\gamma)  \leq C\sigma_r^{-1},
\]
for some constant $C>0.$ 
\end{Proposition} 

\noindent As in the previous sections, even the purely qualitative asymptotic Gaussian behavior for this model does not appear to have been established in the literature. Moreover, the corresponding functional version of the convergence remains an interesting open problem for future research.

\section{Proof of Theorem \ref{thm:application}}\label{sec:secondhtoemre}
By Theorem \ref{thm:mainabstractone}, it suffices to bound from the term above
\begin{align*}
\mathcal{R}_r
  &:=\frac{1}{\sigma_r^3}\int_{\R^d}\E\left[\left|\int_0^r\langle\mu^x_s,\psi\rangle ds\right|^3\right] \nu(dx).
\end{align*}
By elementary computations,
\begin{align*}
\mathcal{R}_r
  &=\sigma_r^{-3}\int_{\R^{d}}\int_{[0,r]^3}\sum_{1\leq i_1,i_2,i_3\leq m}a_{i_1}a_{i_2}a_{i_3}\E[\psi(\mathbf{X}_{s_1}^{\mathbf{y},i_1})\psi(\mathbf{X}_{s_2}^{\mathbf{y},i_2})\psi(\mathbf{X}_{s_3}^{\mathbf{y},i_3})]d\mathbf{s}\nu(d\mathbf{y}).	
\end{align*}
By a symmetry argument we can show that 
\begin{align*}
\mathcal{R}_r
  &\leq A\sigma_r^{-3}\int_{\R^{d}}\int_{[0,r]^3}\sum_{1\leq i_1,i_2,i_3\leq m}\Indi{\{s_{3}\geq s_1\vee s_2\}}a_{i_1}a_{i_2}\E[\psi(\mathbf{X}_{s_1}^{\mathbf{y},i_1})\psi(\mathbf{X}_{s_2}^{\mathbf{y},i_2})\psi(\mathbf{X}_{s_3}^{\mathbf{y},i_3})]d\mathbf{s}\nu(d\mathbf{y}),	
\end{align*}
where $A:=a_1+\dots+a_m$. 
Next we bound the expectation  inside. To this end, it suffices to bound uniformly 
\begin{align*}
\E[\psi(\mathbf{X}_{t_2}^{\mathbf{y},i_3})\ |\ \mathbf{X}_{s}^{\mathbf{y},i_1}, \mathbf{X}_{t_1}^{\mathbf{y},i_2}]
  &=\E[\Theta_{\bf{t},\mathbbm{i}}(\mathbf{X}_{s}^{\mathbf{y},i_1}, \mathbf{X}_{t_1}^{\mathbf{y},i_2})]
\end{align*}
where 
\begin{align*}
\Theta_{\bf{t},\mathbbm{i}}^{\mathbf{y}}(\mathbf{a},\mathbf{b})
  &:=\int_{\R^{d}}\psi(\mathbf{z})f_{s_3|s_1,s_2}^{\mathbbm{i}}(\mathbf{z}|(\mathbf{a},\mathbf{b}))d\mathbf{z}
\end{align*}
By condition \textbf{(H2)}, we obtain the bound 
\begin{align*}
\Theta_{\bf{t},\mathbbm{i}}^{\mathbf{y}}(\mathbf{a},\mathbf{b})
  &\leq \kappa|s_3-(s_1\wedge s_2)|^{-\beta d}\int_{\R^{d}}\psi(\mathbf{z})d\mathbf{z}.
\end{align*}
From here it follows 
\begin{align*}
\mathcal{R}_r
  &\leq A\kappa \sigma_r^{-3}\int_{\R^{d}}\int_{[0,r]^3}s_{3}^{-\beta d}\sum_{1\leq i_1,i_2,i_3\leq m}a_{i_1}a_{i_2}\E[\psi(\mathbf{X}_{s_1}^{\mathbf{y},i_1})\psi(\mathbf{X}_{s_2}^{\mathbf{y},i_2})]d\mathbf{s}\nu(d\mathbf{y}).	
\end{align*}
Integrating, we obtain the desired result
\begin{align*}
\mathcal{R}_r
  &\leq \frac{A\kappa}{(1-\beta d)}\frac{r^{1-d\beta}}{\sigma_r}.	
\end{align*}

\appendix
\section{Auxiliary lemmas}
In this section, we collect the technical lemmas used throughout the paper, providing suitable estimates for the standard deviations of the centered additive functionals.\\

\subsection{Variance estimates in the $\alpha$-stable case} 
We begin with the model introduced in Section~\ref{sec:firstpresentslpha}. 

\begin{lemma}\label{eq:sigmafrombelowstable}
For the model described in Section \ref{sec:firstpresentslpha}, under assumption $d=1$, there exists a constant $c>0$ such that 
\[
\sigma_r^2 \;\geq\; c\, r^{\,2 - d/\alpha}.
\]
\end{lemma}

\begin{proof}
From \eqref{eq:identitysigmar}, we can write 
\[
\sigma_r^2
  = \int_{\mathbb{X}} \int_{\Omega} \big|u_r^{\psi}(x,v)\big|^2 \, \mathbb{Q}(dv)\, \nu(dx).
\]
Using the explicit form of $u_r^{\psi}$ given in (\ref{add-func-one-perticle}), we obtain
\begin{equation}\label{eq:sigmaralpha}
\sigma_r^2
  = \int_{\mathbb{X}} \mathbb{E}\!\left[\,
      \left|\int_{0}^r(a_{1}\psi(x+\mathbf{X}_t^{x,1}) + \cdots + a_{m}\psi(x+\mathbf{X}_t^{x,m}))dt \right|^2
    \,\right] \nu(dx).
\end{equation}
Since the Radon-Nikodym derivative of $\nu$ with respect to the Lebesgue measure is bounded away from zero, there exists $\delta>0$ such that 
\[
\sigma_r^2 \;\geq\; 
  \delta \sum_{i=1}^{m} a_i^2 \int_{\mathbb{X}} \mathbb{E}\!\left[\,\left|\int_0^r\psi(x+\mathbf{X}_t^{x,i})dt\right|^2 \,\right] dx.
\]
By comparison with \eqref{eq:sigmaralpha}, the right-hand side coincides with the variance of an additive functional for a model of the type considered in \cite[Theorem~2.1]{BGT-I}. This yields the desired lower bound.
\end{proof}

\subsection{Variance estimates in the diffusive case} 
We now turn to the diffusive setting introduced in Section~\ref{eq:sec:diffusivecaseexamples}. 

\begin{lemma}\label{eq:sigmafrombelowdiffusion}
Suppose that the test function $\psi$ is bounded from below in a compact set $\mathcal{K}$ with non-empty interior. Under the assumption $d=1$, there exists a constant $c>0$ such that
\[
\sigma_r^2 \;\geq\; c\, r^{2 - d/2}.
\]
\end{lemma}
\begin{proof}
The argument follows the same general strategy as in the proof of Lemma~\ref{eq:sigmafrombelowstable}. One first observes that there exists a constant $\delta>0$ such that 
\[
\sigma_r^2 \;\geq\; 
  \delta \sum_{i=1}^{m} a_i^2 \int_{\mathbb{X}} 
     \mathbb{E}\!\left[\,\Big|\int_0^{r}\psi(x+\mathbf{X}_t^{x,i})\,dt\Big|^2 \,\right] dx.
\]
Since $\psi$ is assumed to be non-trivial on some compact set $\mathcal{K}$ of positive Lebesgue measure, we may further bound the right-hand side by replacing $\psi$ with the indicator of $\mathcal{K}$. In this way we obtain 
\[
\sigma_r^2 \;\geq\; 
  \delta \sum_{i=1}^{m} a_i^2 \int_{\mathbb{X}} 
     \mathbb{E}\!\left[\,\Big|\int_0^{r}\mathbf{1}_{\mathcal{K}}(x+\mathbf{X}_t^{x,i})\,dt\Big|^2 \,\right] dx.
\]
The expectation inside the integral can be expressed, via a standard symmetrization argument, as a double integral over the time square $[0,r]^2$ involving the joint probabilities that the diffusion visits $\mathcal{K}$ at two times $t_1\leq t_2$. More precisely,
\[
\mathbb{E}\!\left[\,\Big|\int_0^{r}\mathbf{1}_{\mathcal{K}}(x+\mathbf{X}_t^{x,i})\,dt\Big|^2 \,\right] 
= 2 \int_{[0,r]^2}\mathbf{1}_{\{t_1\leq t_2\}} \,
   \mathbb{P}\!\left[x+\mathbf{X}_{t_1}^{x,i}\in \mathcal{K},\,x+\mathbf{X}_{t_2}^{x,i}\in \mathcal{K}\right]dt_1dt_2.
\]
In order to bound this probability from below, we write it in terms of the transition density $p_t(x,y)$ of the underlying diffusion. By the Markov property we have
\[
\mathbb{P}\!\left[x+\mathbf{X}_{t_1}^{x,i}\in \mathcal{K},\,x+\mathbf{X}_{t_2}^{x,i}\in \mathcal{K}\right]
= \int_{\mathcal{K}} \int_{\mathcal{K}} p_{t_1}(x,y)\, p_{t_2-t_1}(y,z)\, dz\, dy.
\]
Using the standard Gaussian lower bound for transition densities given (\ref{bounded-density-diffusion}), this expression can itself be bounded from below by a constant multiple of the same double integral with $p$ replaced by the Gaussian heat kernel, which in the sequel will be denoted by $p^\gamma$. In particular, there exists $\delta'>0$ such that 
\[
\int_{\mathcal{K}} \int_{\mathcal{K}} p_{t_1}(x,y)\, p_{t_2-t_1}(y,z)\, dz\, dy 
\;\geq\; \delta' \int_{\mathcal{K}} \int_{\mathcal{K}} p_{t_1}^{\gamma}(x,y)\, p_{t_2-t_1}^{\gamma}(y,z)\, dz\, dy.
\]
This shows that the variance $\sigma_r^2$ is bounded from below by a positive constant times the variance of an additive functional driven by Brownian motion with test function $\varphi=\mathbf{1}_{\mathcal{K}}$, namely,
\[
\sigma_r^2 \;\geq\; 
  \delta\delta' \sum_{i=1}^{m} a_i^2 \int_{\mathbb{X}}  
  \mathbb{E}\!\left[\,\Big|\int_0^{r}\varphi(x+\mathbf{W}_t)\,dt\Big|^2 \,\right] dx.
\]
Finally, by comparison with \eqref{eq:sigmaralpha}, this coincides with the variance of an additive functional for a Brownian model of the type treated in \cite[Theorem~2.1]{BGT-I}, from which the claimed lower bound follows.
\end{proof}

\subsection{Variance estimates in the fractional Brownian motion  case} 
We now turn to the setting in which the underlying process is a fractional Brownian motion  with Hurst parameter $H$ in $(0,1)$. 

\begin{lemma}\label{eq:sigmafrombelowfBm}
For the model described in Section~\ref{eq:fractionalbrownianmotionparticles}, when $\mathbf{X}^{i,x}$ is a fractional Brownian motion and $Hd<1$, there exists a constant $c>0$ such that 
\[
\sigma_r^2 \;\geq\; c\, r^{2(1-H)}\Big[
     r^{\frac{3}{2}H}\Indi{\{d=1\}}
   + r^H\log(r)\Indi{\{d=2\}}
   + r^H\Indi{\{d\geq 3\}}
\Big].
\]
\end{lemma}

\begin{proof}
Proceeding as in the diffusive case, we can show the existence of a compact set $\mathcal{K}$ with non-trivial Lebesgue measure, and a positive constant $\delta>0$, such that 
\[
\sigma_r^2 \;\geq\; 
  \delta \sum_{i=1}^{m} a_i^2 \int_{\mathbb{X}} 
     \mathbb{E}\!\left[\,\Big|\int_0^{r}\Indi{\mathcal{K}}(x+\mathbf{X}^{H,i}_t)\,dt\Big|^2 \,\right] dx.
\]
By the same symmetrization argument as before, 
\begin{align*}
\mathbb{E}\!\left[\,\Big|\int_0^{r}\Indi{\mathcal{K}}(x+\mathbf{B}^H_t)\,dt\Big|^2 \,\right] 
   &= 2 \int_{[0,r]^2}\Indi{\{t_1\leq t_2\}} \,
   \mathbb{P}\!\left[x+\mathbf{B}^H_{t_1}\in \mathcal{K},\,x+\mathbf{B}^H_{t_1}+(\mathbf{B}^H_{t_2}-\mathbf{B}^H_{t_1})\in \mathcal{K}\right]dt_1dt_2.	
\end{align*}
If we define $\varphi_x(a,b):=\Indi{\{x+a\in\mathcal{K}\}}\Indi{\{x+a+b\in\mathcal{K}\}}$, we then conclude that 
\[
\mathbb{E}\!\left[\,\Big|\int_0^{r}\Indi{\mathcal{K}}(x+\mathbf{X}_t^{x,i})\,dt\Big|^2 \,\right] 
= 2 \int_{[0,r]^2}\Indi{\{t_1\leq t_2\}} \,
   \mathbb{E}\!\left[\varphi_x(\mathbf{X}_{t_1}^{x,i}
   , \mathbf{X}_{t_2}^{x,i}-\mathbf{X}_{t_1}^{x,i})\right]dt_1dt_2.
\]
In order to bound this probability from below, we observe that  the argument of $\varphi_x$ is the  centered bivariate Gaussian $\big(\mathbf{X}_{t_1}^{x,i},\,\mathbf{X}_{t_2-t_1}^{x,i}\big)$  whose joint density is bounded from below by the joint density of the random vector $\big(t_{1}^{H}\mathbf{N}_1,(t_{2}-t_1)^{H}\mathbf{N}_2\big)/\sqrt{2}$, where the ${\mathbf{N}}_{j}$'s are independent standard Gaussian vectors.
% From here we conclude that  , with covariance matrix
%\[
%\Sigma_\Delta(t_1,t_2)=
%\begin{pmatrix}
%t_1^{2H} & c_\Delta \\
%c_\Delta & (t_2-t_1)^{2H}
%\end{pmatrix},
%\]
%for $c_\Delta$ given by $c_\Delta=\tfrac12\!\left(t_2^{2H}-t_1^{2H}-(t_2-t_1)^{2H}\right)$. By the Loewner inequalities, the following inequalities hold with respect to the positive definite order of matrices
%\[
%\Sigma_\Delta(t_1,t_2)\;\leq \;
%2 \mathrm{diag}(t_1^{2H},(t_2-t_1)^{2H}),
%\]
%Using this comparison, one can easily show that that the joint density 
%of $\big(\mathbf{X}_{t_1}^{x,i},\,\mathbf{X}_{t_2}^{x,i}-\mathbf{X}_{t_1}^{x,i}\big)$ is bounded from below by the joint density of 
%$\big(t_{1}^{H}\mathbf{N}_1,(t_{2}-t_1)^{H}\mathbf{N}_2\big)/\sqrt{2}$, where the ${\mathbf{N}}_{j}$ are independent standard Gaussian vectors.
From here we conclude that 
\begin{align*}
\mathbb{E}\!\left[\,\Big|\int_0^{r}\Indi{\mathcal{K}}(x+\mathbf{X}_t^{x,i})\,dt\Big|^2 \,\right] 
\geq  2 \int_{[0,r]^2}\Indi{\{t_1\leq t_2\}} \,
   \mathbb{E}\!\left[\tilde{\varphi}_x\big(t_{1}^{H}\mathbf{N}_1,(t_{2}-t_1)^{H}\mathbf{N}_2\big)\right]dt_1dt_2.	
\end{align*}
where $\tilde{\varphi}$ denotes the dilation of $\varphi$ by the factor $1/\sqrt{2}.$ Changing the coordinates $(t_1,t_2)$ by $t_1^{H},t_1^{H}+(t_2-t_1)^{H})$, we then obtain 
\begin{multline*}
\mathbb{E}\!\left[\,\Big|\int_0^{r}\Indi{\mathcal{K}}(x+\mathbf{X}_t^{x,i})\,dt\Big|^2 \,\right]\\ 
\geq   \frac{2}{H^2} \!\!\int_{[0,r^H]^2} \Indi{\{t_1\leq t_2\}}
   \mathbb{E}\!\left[\tilde{\varphi}_x\big(t_{1}\mathbf{N}_1,(t_{2}-t_1)\mathbf{N}_2\big)\right]\,
 t_1^{\frac{1}{H}-1}(t_2-t_1)^{\frac{1}{H}-1}\,dt_2\,dt_1.
\end{multline*}
Then, we observe that
\[
\begin{aligned}
\mathbb{E}\!\left[\,\Big|\int_0^{r}\mathbf{1}_{\mathcal{K}}(x+\mathbf{X}_t^{x,i})\,dt\Big|^2 \,\right] 
&\ge \frac{2}{H^2}\!\!\int_{[0,r^{H}]^2}
\mathbf{1}_{\{t_1\le t_2\}}
\mathbf{1}_{\{t_1\ge r^{H}/2\}}
\mathbf{1}_{\{t_2-t_1\ge r^{H}/2\}} \\
&\hspace{6em}\times
\mathbb{E}\!\left[\tilde{\varphi}_x\!\big(t_{1}\mathbf{N}_1,(t_{2}-t_1)\mathbf{N}_2\big)\right]\,
t_1^{\frac{1}{H}-1}(t_2-t_1)^{\frac{1}{H}-1}\,dt_2\,dt_1 \\[4pt]
\end{aligned}
\]
\noindent Observe that on the region 
\[
\{\,0\le t_1\le r^{H},\ t_1\le t_2\le r^{H},\ t_1\ge r^{H}/2,\ t_2-t_1\ge r^{H}/2\,\},
\]
we have $t_1\ge r^{H}/2$ and $t_2-t_1\ge r^{H}/2$, so that
\[
t_1^{\frac{1}{H}-1}(t_2-t_1)^{\frac{1}{H}-1}
\;\ge\;\Big(\tfrac{r^{H}}{2}\Big)^{\frac{1}{H}-1}\Big(\tfrac{r^{H}}{2}\Big)^{\frac{1}{H}-1}
\;=\;2^{-2(\frac{1}{H}-1)}\,r^{2(1-H)}.
\]
This yields 

\begin{equation}
    \mathbb{E}\!\left[\,\left|\int_0^{r}\mathbf{1}_{\mathcal{K}}(x+\mathbf{X}_t^{x,i})\,dt\right|^2 \,\right]
\geq \frac{2}{H^2}\,2^{-2(\frac{1}{H}-1)}\,r^{2(1-H)}I_r[F],
\end{equation}
%\!\!\int_{[0,r^{H}]^2}
%\mathbf{1}_{\{t_1\le t_2\}}
%\mathbf{1}_{\{t_1\ge r^{H}/2\}}
%\mathbf{1}_{\{t_2-t_1\ge r^{H}/2\}} F(t_1,t_2)dt_2\,dt_1.
where $I_r$ is the operator defined through  
\begin{align*}
I_r[F]
:= \int_{[0,r^{H}]^2} 
   \mathbf{1}_{\{t_1 \le t_2\}}
   \mathbf{1}_{\{t_1 \ge r^{H}/2\}}
   \mathbf{1}_{\{t_2 - t_1 \ge r^{H}/2\}}
   F(t_1,t_2)\,dt_1\,dt_2,
\end{align*}
with $F$  given by 
\begin{align*}
F(t_1,t_2)
  &:=	\mathbb{E}\!\left[\tilde{\varphi}_x\!\big(t_{1}\mathbf{N}_1,(t_{2}-t_1)\mathbf{N}_2\big)\right].
\end{align*}
Using the change of variables $s := t_1 - \tfrac{r^{H}}{2}$ and $\tau := t_2 - t_1 - \tfrac{r^{H}}{2}$, we get 
\begin{align*}
I_r[F] 
= \int_{[0,r^{H}/2]^2} 
   F\!\big(s+\tfrac{r^{H}}{2},\,s+\tau+r^H\big)\,ds\,d\tau.
\end{align*}
Now make a second change of variables, $u := s, \qquad v := s+\tau$, which  maps $[0,r^{H}/2]^2$ into the region
\begin{align}\label{eq:domainaux}
\{(u,v):\, 0 \le u \le r^{H}/2,\ \ u \le v \le u + r^{H}/2\},	
\end{align}
leading to the identity 
\begin{align*}
I_r[F] 
= \int_{0}^{r^{H}/2}\!\int_{u}^{u+r^{H}/2} 
   F\!\big(u+\tfrac{r^{H}}{2},\,v+{r^H}\big)\,dv\,du.
\end{align*}
Observe that the square 
$[0,r^{H}/2]^2 \cap \{u \le v\}$ 
is contained in the domain \eqref{eq:domainaux}, so we deduce that 
\begin{align*}
I_r[F]  &\ge \int_{[0,r^{H}/2]^2} \mathbf{1}_{\{u \le v\}}\,
   F\!\big(u+\tfrac{r^{H}}{2},\,v+{r^H}\big)\,dv\,du.
\end{align*}
Changing back to the coordinates $(t_1,t_2){=(u+r^H/2,v+r^H)}$ we obtain the final inequality 
\begin{multline*}
\mathbb{E}\!\left[\,\Big|\int_0^{r}\mathbf{1}_{\mathcal{K}}(x+\mathbf{X}_t^{x,i})\,dt\Big|^2 \,\right]\\
\ge \frac{2^{-2/H+3}}{H^2}\,r^{2(1-H)}
\!\!\int_{[0,r^{H}/2]^2}
\mathbf{1}_{\{t_1\le t_2\}}\mathbb{E}\!\left[\tilde{\varphi}_x\!\big(t_{1}\mathbf{N}_1,(t_{2}-t_1)\mathbf{N}_2\big)\right]dt_2\,dt_1.	
\end{multline*}
	Next we identify the right hand side as the variance $\sigma_r^2$ of the additive functional driven by Brownian motion with test function $\tilde{\varphi}$. The result then follows from the estimates from \cite[Theorem 2.1.]{BGT-I} \cite[Theorem 2.1]{BGT-II}, the integral in the right-hand side is of the order 
\begin{align*}
	r^{\frac{3}{2}H}\Indi{\{d=1\}}+ r^H\log(r)\Indi{\{d=2\}} +r^H\Indi{\{d\geq 3\}}.
	\end{align*}
  %  {\color{cyan}
%\begin{align*}
%\Indi{\{d=1\}}	(r^H)^{3/2}+\Indi{\{d=2\}} r^H\log(r) +\Indi{\{d\geq 3\}}r^H
%	\end{align*}}
Thus, there exist a positive constant $c$ such that $$\sigma_r^2\geq cr^{2(1-H)}\left[ r^{\frac{3}{2}H}\Indi{\{d=1\}}+ r^H\log(r)\Indi{\{d=2\}} +r^H\Indi{\{d\geq 3\}}\right],$$which gives the  result. 
\end{proof}

\noindent \textbf{Acknowledgements}\\
Arturo Jaramillo Gil was supported by
the grant CBF2023-2024-2088. Antonio Murillo-Salas thanks the University of Guanajuato for the sabbatical leave.

\bibliographystyle{plain}  % or another style like alpha, abbrv, etc.
\bibliography{Bibliography}

@book{NourdinFBM,
  author    = {Nourdin, Ivan},
  title     = {Selected Aspects of Fractional Brownian Motion},
  series    = {SpringerBriefs in Probability and Mathematical Statistics},
  publisher = {Springer},
  address   = {Cham},
  year      = {2012},
  doi       = {10.1007/978-3-642-28555-5},
  isbn      = {978-3-642-28554-8}
}

@article{Mecke1967,
  author    = {Mecke, J.},
  title     = {Station{\"a}re zuf{\"a}llige Maße auf lokalkompakten Abelschen Gruppen},
  journal   = {Zeitschrift f{\"u}r Wahrscheinlichkeitstheorie und Verwandte Gebiete},
  volume    = {9},
  pages     = {36--58},
  year      = {1967},
  doi       = {10.1007/BF00537588}
}

@book{LastPenrose2017,
  author    = {Last, G{\"u}nter and Penrose, Mathew D.},
  title     = {Lectures on the Poisson Process},
  publisher = {Cambridge University Press},
  series    = {Institute of Mathematical Statistics Textbooks},
  volume    = {7},
  year      = {2017},
  isbn      = {978-1107160513},
  doi       = {10.1017/9781316104477}
}

@book{PeccatiReitzner2016,
  title     = {Stochastic Analysis for Poisson Point Processes: Malliavin Calculus, Wiener–Ito Chaos Expansions and Stochastic Geometry},
  editor    = {Peccati, Giovanni and Reitzner, Matthias},
  publisher = {Springer International Publishing},
  series    = {Bocconi, Springer Series},
  volume    = {7},
  year      = {2016},
  isbn      = {978-3-319-05232-8},  
  isbn-ebook= {978-3-319-05233-5},
  doi       = {10.1007/978-3-319-05233-5}
}

@book{Mishura2008,
  author    = {Yulia Mishura},
  title     = {Stochastic Calculus for Fractional Brownian Motion and Related Processes},
  series    = {Lecture Notes in Mathematics},
  volume    = {1929},
  publisher = {Springer},
  address   = {Berlin},
  year      = {2008},
  doi       = {10.1007/978-3-540-75873-0}
}

@book{Spohn2012,
  author    = {Herbert Spohn},
  title     = {Large Scale Dynamics of Interacting Particles},
  series    = {Texts and Monographs in Physics},
  publisher = {Springer},
  address   = {Berlin},
  year      = {2012}
}

@article{Dyson1962,
  author    = {Dyson, Freeman J.},
  title     = {A Brownian-motion model for the eigenvalues of a random matrix},
  journal   = {Journal of Mathematical Physics},
  volume    = {3},
  number    = {6},
  pages     = {1191--1198},
  year      = {1962},
  doi       = {10.1063/1.1703862}
}

@book{AndersonGuionnetZeitouni2010,
  author    = {Anderson, Greg W. and Guionnet, Alice and Zeitouni, Ofer},
  title     = {An Introduction to Random Matrices},
  series    = {Cambridge Studies in Advanced Mathematics},
  volume    = {118},
  publisher = {Cambridge University Press},
  year      = {2010},
  doi       = {10.1017/CBO9780511801334}
}

@book{Mehta2004,
  author    = {Mehta, Madan Lal},
  title     = {Random Matrices},
  edition   = {3rd},
  series    = {Pure and Applied Mathematics},
  volume    = {142},
  publisher = {Elsevier/Academic Press},
  address   = {Amsterdam},
  year      = {2004},
  isbn      = {978-0-12-088409-4}
}

@book{Kallenberg2017,
  author    = {Kallenberg, Olav},
  title     = {Random Measures, Theory and Applications},
  series    = {Probability Theory and Stochastic Modelling},
  volume    = {77},
  publisher = {Springer},
  address   = {Cham},
  year      = {2017},
  doi       = {10.1007/978-3-319-41598-7},
  isbn      = {978-3-319-41597-0},
  pages     = {xvii+944}
}

@book{DVJ2008,
  author    = {Daley, Daryl J. and Vere-Jones, David},
  title     = {An Introduction to the Theory of Point Processes. Vol. I and II},
  edition   = {2nd},
  publisher = {Springer},
  address   = {New York},
  year      = {2003--2008},
  series    = {Probability and Its Applications},
  doi       = {10.1007/978-0-387-49835-5},
  isbn      = {978-0-387-49834-8},
  pages     = {xvi+469 (Vol. I), xvii+592 (Vol. II)}
}

@incollection{Dawson1993,
  author    = {Dawson, Donald A.},
  title     = {Measure-valued {M}arkov processes},
  booktitle = {\'Ecole d'\'Et\'e de Probabilit\'es de Saint-Flour XXI---1991},
  series    = {Lecture Notes in Mathematics},
  volume    = {1541},
  pages     = {1--260},
  publisher = {Springer},
  address   = {Berlin},
  year      = {1993},
  doi       = {10.1007/BFb0084190}
}

@book{Etheridge2000,
  author    = {Etheridge, Alison M.},
  title     = {An Introduction to Superprocesses},
  series    = {University Lecture Series},
  volume    = {20},
  publisher = {American Mathematical Society},
  address   = {Providence, RI},
  year      = {2000},
  isbn      = {0-8218-2045-5},
  pages     = {x+206}
}

@incollection{Perkins2002,
  author    = {Perkins, Edwin A.},
  title     = {Dawson--Watanabe superprocesses and measure-valued diffusions},
  booktitle = {\'Ecole d'\'Et\'e de Probabilit\'es de Saint-Flour XXIX---1999},
  series    = {Lecture Notes in Mathematics},
  volume    = {1781},
  pages     = {125--324},
  publisher = {Springer},
  address   = {Berlin},
  year      = {2002},
  doi       = {10.1007/3-540-45547-7_2}
}

@book{Dynkin1991,
  author    = {Dynkin, Eugene B.},
  title     = {Superprocesses and Partial Differential Equations},
  publisher = {Princeton University Press},
  address   = {Princeton, NJ},
  series    = {Annals of Mathematics Studies},
  volume    = {146},
  year      = {1991},
  isbn      = {978-0-691-04665-2},
  pages     = {x+234}
}

@incollection{Xiao2009SLND,
  author    = {Yimin Xiao},
  title     = {Sample Path Properties of Anisotropic Gaussian Random Fields},
  booktitle = {A Minicourse on Stochastic Partial Differential Equations},
  editor    = {Davar Khoshnevisan and Firas Rassoul{-}Agha},
  series    = {Lecture Notes in Mathematics},
  volume    = {1962},
  pages     = {145--212},
  publisher = {Springer},
  address   = {Berlin, Heidelberg},
  year      = {2009},
  doi       = {10.1007/978-3-540-85994-9_5}
}

@book {MR1280932,
    AUTHOR = {Samorodnitsky, Gennady and Taqqu, Murad S.},
     TITLE = {Stable non-{G}aussian random processes},
    SERIES = {Stochastic Modeling},
      NOTE = {Stochastic models with infinite variance},
 PUBLISHER = {Chapman \& Hall, New York},
      YEAR = {1994},
     PAGES = {xxii+632},
      ISBN = {0-412-05171-0},
   MRCLASS = {60E07 (60G18 60G99 60H99 62J05 62M10 62M15)},
  MRNUMBER = {1280932},
MRREVIEWER = {Thomas\ Mikosch},
}

@article{neufcourt2016third,
  title     = {A third-moment theorem and precise asymptotics for variations of stationary Gaussian sequences},
  author    = {Neufcourt, L{\'e}o and Viens, Fr{\'e}d{\'e}ri},
  journal   = {ALEA, Latin American Journal of Probability and Mathematical Statistics},
  volume    = {13},
  pages     = {239--264},
  year      = {2016},
  doi       = {10.30757/ALEA.v13-10},
  url       = {https://alea.impa.br/articles/v13/13-10.pdf}
}

@article{dobler2018fourth,
  title     = {The fourth moment theorem on the Poisson space},
  author    = {D{\"o}bler, Christian and Peccati, Giovanni},
  journal   = {The Annals of Probability},
  volume    = {46},
  number    = {4},
  pages     = {1878--1916},
  year      = {2018},
  doi       = {10.1214/17-AOP1215},
  url       = {https://doi.org/10.1214/17-AOP1215}
}

@article{nualart2005central,
  title     = {Central limit theorems for sequences of multiple stochastic integrals},
  author    = {Nualart, David and Peccati, Giovanni},
  journal   = {Annals of Probability},
  volume    = {33},
  number    = {1},
  pages     = {177--193},
  year      = {2005},
  doi       = {10.1214/009117904000000621},
  url       = {https://projecteuclid.org/euclid.aop/1108674681}
}

@book{nourdin2012normal,
  title     = {Normal Approximations with Malliavin Calculus: From Stein's Method to Universality},
  author    = {Nourdin, Ivan and Peccati, Giovanni},
  year      = {2012},
  publisher = {Cambridge University Press},
  series    = {Cambridge Tracts in Mathematics},
  volume    = {192},
  doi       = {10.1017/CBO9781139084659}
}

@article{nourdin2009stein,
  title     = {Stein's method on Wiener chaos},
  author    = {Nourdin, Ivan and Peccati, Giovanni},
  journal   = {Probability Theory and Related Fields},
  volume    = {145},
  number    = {1},
  pages     = {75--118},
  year      = {2009},
  doi       = {10.1007/s00440-008-0162-x}
}

@article{kemp2012wigner,
  title     = {Wigner chaos and the fourth moment},
  author    = {Kemp, Todd and Nourdin, Ivan and Peccati, Giovanni and Speicher, Roland},
  journal   = {The Annals of Probability},
  volume    = {40},
  number    = {4},
  pages     = {1577--1635},
  year      = {2012},
  doi       = {10.1214/11-AOP657},
  url       = {https://projecteuclid.org/euclid.aop/1342790122}
}

@article{arizmendi2021convergence,
  title     = {Convergence of the Fourth Moment and Infinite Divisibility: Quantitative Estimates},
  author    = {Arizmendi, Octavio and Jaramillo, Arturo},
  journal   = {arXiv preprint arXiv:2103.06129},
  year      = {2021},
  url       = {https://arxiv.org/abs/2103.06129}
}

@article{C-G,
  author    = {Cox, J. Theodore and Griffeath, David},
  title     = {Occupation Times for Critical Branching Brownian Motions},
  journal   = {Annals of Probability},
  volume    = {13},
  number    = {4},
  pages     = {1108--1132},
  year      = {1985},
  doi       = {10.1214/aop/1176992799},
  url       = {https://doi.org/10.1214/aop/1176992799}
}

@article{D-W,
  author    = {Deuschel, Jean-Dominique and Wang, Kongming},
  title     = {Large deviations for the occupation time functional of a Poisson system of independent Brownian particles},
  journal   = {Stochastic Processes and their Applications},
  volume    = {52},
  number    = {2},
  pages     = {183--209},
  year      = {1994},
  issn      = {0304-4149},
  doi       = {10.1016/0304-4149(94)90024-8},
  url       = {https://doi.org/10.1016/0304-4149(94)90024-8}
}

@article{BGT-I,
  author    = {Bojdecki, Tomasz and Gorostiza, Luis G. and Talarczyk, Agnieszka},
  title     = {Limit theorems for occupation time fluctuations of branching systems I: Long-range dependence},
  journal   = {Stochastic Processes and their Applications},
  volume    = {116},
  number    = {1},
  pages     = {19--35},
  year      = {2006}
}

@article{BGT-II,
  author    = {Bojdecki, Tomasz and Gorostiza, Luis G. and Talarczyk, Agnieszka},
  title     = {Limit theorems for occupation time fluctuations of branching systems II: Critical and large dimensions},
  journal   = {Stochastic Processes and their Applications},
  volume    = {116},
  number    = {1},
  pages     = {1--18},
  year      = {2006}
}

@article{LM-MS-RG,
  author    = {L{\'o}pez-Mimbela, Jos{\'e} A. and Murillo-Salas, Antonio and Ram{\'i}rez-Gonz{\'a}lez, Jos{\'e} H.},
  title     = {Occupation time fluctuations of an age-dependent critical binary branching particle system},
  journal   = {ALEA, Latin American Journal of Probability and Mathematical Statistics},
  volume    = {21},
  year      = {2024},
  doi       = {10.30757/ALEA.v21-24},
  url       = {https://alea.impa.br/articles/v21/21-24.pdf}
}

@incollection {Stroock,
    AUTHOR = {Stroock, Daniel W.},
     TITLE = {Diffusion semigroups corresponding to uniformly elliptic
              divergence form operators},
 BOOKTITLE = {S\'eminaire de {P}robabilit\'es, {XXII}},
    SERIES = {Lecture Notes in Math.},
    VOLUME = {1321},
     PAGES = {316--347},
 PUBLISHER = {Springer, Berlin},
      YEAR = {1988},
      ISBN = {3-540-19351-0},
   MRCLASS = {35J25 (47D05 60J35)},
  MRNUMBER = {960535},
MRREVIEWER = {J.\ A.\ Goldstein},
       DOI = {10.1007/BFb0084145},
       URL = {https://doi.org/10.1007/BFb0084145},
}

@article {C-G1984,
    AUTHOR = {Cox, J. Theodore and Griffeath, David},
     TITLE = {Large deviations for {P}oisson systems of independent random
              walks},
   JOURNAL = {Z. Wahrsch. Verw. Gebiete},
  FJOURNAL = {Zeitschrift f\"ur Wahrscheinlichkeitstheorie und Verwandte
              Gebiete},
    VOLUME = {66},
      YEAR = {1984},
    NUMBER = {4},
     PAGES = {543--558},
      ISSN = {0044-3719},
   MRCLASS = {60K35 (60F10 60J15)},
  MRNUMBER = {753813},
MRREVIEWER = {S.\ C.\ Port},
       DOI = {10.1007/BF00531890},
       URL = {https://doi.org/10.1007/BF00531890},
}
\end{document}